\documentclass{article}
\usepackage{amsmath,amsfonts,amssymb,amsthm}
\usepackage{latexsym}
\usepackage{epsfig,overpic}
\usepackage{stmaryrd}
\usepackage{color}
\usepackage{graphicx}

\newtheorem{lemma}{Lemma}[section]
\newtheorem{example}{Example}

\newtheorem{theorem}{Theorem}[section]

\newtheorem{proposition}{Proposition}[section]
\providecommand{\keywords}[1]{\textbf{\textit{Keywords:}} #1}

\let\ep\varepsilon
\let\eps\varepsilon
\def\ds{\mathrm{d}\mathbf{s}}
\newcommand{\R}{\mathbb R}

\newcommand{\bH}{\mathbf H}

\newcommand{\bP}{\mathbf P}

\newcommand{\bn}{\mathbf n}

\newcommand{\bp}{\mathbf p}
\newcommand{\br}{\mathbf r}
\newcommand{\bs}{\mathbf s}

\newcommand{\blf}{\mathbf f}
\newcommand{\bw}{\mathbf w}
\newcommand{\bx}{\mathbf x}

\newcommand{\F}{\mathcal F}

\newcommand{\Div}{\mbox{\rm div}}

\newcommand{\unablah}{{\nabla}_{\Gamma_h}}

\newcommand{\rd}{\mathrm{d}}

\definecolor{darkred}{rgb}{.7,0,0}

\definecolor{green}{rgb}{0,0.7,0}

\def\d{\hspace{2pt}{\rm d}}

\def\nat{\nabla_{\Gamma}}
\def\nath{\nabla_{\Gamma_h}}
\def\ipue{(I_h u^e)|_{\Gamma_h}}

\begin{document}

\title{An adaptive octree finite element method for PDEs posed on surfaces}
\author{
Alexey Y. Chernyshenko\thanks{
Institute of Numerical Mathematics, Russian Academy of Sciences, Moscow 119333}
\and
Maxim A. Olshanskii\thanks{Department of Mathematics, University of Houston, Houston, Texas 77204-3008 {\tt molshan@math.uh.edu}}
}

\maketitle

\markboth{}{}

\begin{abstract}
The paper develops a finite element method for  partial differential equations
posed on hypersurfaces in $\mathbb{R}^N$, $N=2,3$. The method uses traces of bulk finite element
functions on a surface embedded in a volumetric domain. The bulk finite element space is defined on
an octree grid which is locally refined or coarsened  depending on error indicators and estimated
values of the surface curvatures. The cartesian structure of the bulk mesh leads to easy and efficient
adaptation process, while the trace finite element method makes fitting the mesh to the surface
unnecessary. The number of degrees of freedom involved in computations is consistent with the
two-dimension nature of surface PDEs.
No parametrization of the surface is required; it can be given implicitly by a level set function. In practice, a variant of the marching cubes method is used to recover the surface with the second order accuracy.  
We prove the optimal order of accuracy for the trace finite element method in $H^1$ and $L^2$ surface norms for a problem with smooth solution and quasi-uniform mesh refinement.
Experiments with less regular problems demonstrate optimal convergence with respect to the number of degrees of freedom,
if  grid adaptation is based on an appropriate error indicator.  The paper shows results of numerical experiments for
a variety of geometries and problems, including advection-diffusion equations on surfaces.
Analysis and numerical results of the paper suggest that combination of cartesian adaptive meshes and the  unfitted (trace) finite elements provide simple, efficient, and reliable tool for numerical treatment of PDEs posed on surfaces.
 \end{abstract}




\section{Introduction}
Partial differential equations posed on surfaces arise in mathematical models for many natural phenomena: diffusion along grain boundaries \cite{grain2}, lipid interactions in biomembranes \cite{membrains1}, and transport of surfactants on multiphase flow interfaces \cite{GrossReuskenBook}, as well as in many engineering and bioscience applications: vector field visualization \cite{vector}, textures synthesis \cite{texture1}, brain warping \cite{imaging2}, fluids in lungs \cite{lungs} among others.
Thus, recently there has been a significant increase of interest  in developing and analyzing numerical methods for  PDEs on surfaces.

One natural approach to solving PDEs on surfaces numerically is based on surface triangulation.  In this class of methods, one typically assumes that a parametrization of a surface is given and the surface is approximated by a family of consistent  regular triangulations. It is common to assume that all nodes of the triangulations lie on the surface. The analysis of a finite element method based on surface triangulations was first done in~\cite{Dziuk88}. To avoid surface triangulation and remeshing (if the surface  evolves), another approach was taken in \cite{BCOS01}: It was proposed to extend a partial differential equation from the surface to a set of positive Lebesgue measure in $\R^N$. The resulting PDE is then solved in one dimension higher, but can be solved on a mesh that is unaligned to the surface.
A surface is allowed  to be defined implicitly as a zero set of  a given level set function. However, the resulting bulk elliptic or parabolic equations are degenerate, with no diffusion acting  in the direction normal to the surface. Versions of the method, where only an $h$-narrow band around the surface is used to define a finite element method, were studied in  \cite{Alg2,OlshSafin}. An overview of finite element methods for surface PDEs and more references can be found in the recent review paper~\cite{DE2013}.

Another unfitted finite element method for  equations posed on surfaces was introduced in~\cite{ORG09,OlsR2009}. That method does not use an extension of the surface partial differential equation.  It is instead based on a restriction (trace) of the outer  finite element spaces to a surface. The trace finite element method does not need parametrization or fitting a mesh
to the surface and avoids some well known pitfalls of PDE-extension based methods related to the bulk equation degeneracy and numerical boundary conditions. Since only those bulk elements are involved in computations which are intersected by a surface, the number of active degrees of freedom is consistent with the dimensionality of the surface problem. The trace finite element method is also very natural approach when one needs to solve a system of partial differential equations that couples bulk domain effects with interface (or surface) effects, the situation which occurs in a number of applications~\cite{Bonito,ER2013}. Therefore, recently several authors developed the trace finite element method in various directions: In \cite{Grande,HLZ2013,ORXsinum,ORsinum2014} the method was extended and analysed for the case of evolving surfaces; Papers \cite{BHLZ,GOR2014} considered surface-bulk coupled problems and \cite{ORXimanum} treated singular-perturbed surface advection-diffusion equation; An analysis of  a higher order trace finite elements was given in \cite{Reusken2014}; A posteriori estimates and adaptivity were studied in \cite{DemlowOlsh}; Versions of the method with improved algebraic properties  were introduced in \cite{Alg1,Alg2}. All these studies considered continuous piecewise polynomial (typically $P_1$) bulk finite elements with respect to a regular tetrahedral subdivision of a volumetric domain or a regular triangulation in 2D case.

In the present paper, the trace finite element method is developed for octree bulk meshes.
The cartesian structure and embedded hierarchy of octree grids makes mesh adaptation, reconstruction and data access
fast and easy, which is not always the case for tetrahedral meshes. For these reasons, octree meshes became a common tool
in  image processing, the visualization of amorphous medium, free surface and multi-phase flows computations
and other applications where non-trivial geometries occur, see, for example, \cite{Losasso:04,Meagher,Szeliski,Pop09,Strain,ViscPlast}. Thus, employing octree grids for numerical solution of PDEs on surfaces one benefits from their local adaptation properties  in the case fine surface structures or solution
with a singularity. Moreover, the resulting tool for solving surface PDEs  is ready for coupling with many of existing 
octree-based methods (not restricted to finite element methods) for solving bulk problems. Some of these octree-based solvers are parts of the publicly available software, e.g., \cite{DEAL2,Pop03}.
One intrinsic property of octree grids, however, is only
the first order approximation of curved geometries. It appears that the trace finite element method is the right tool
to deal with this potential shortcoming. We shall see that a surface may cut the octree mesh in an arbitrary way.
The finite element method is unfitted and uses a second order surface recovery with a variant of the marching cubes algorithm~\cite{MCM}. As a formal demonstration of the method accuracy, we  prove the $O(h)$ error estimate in the $H^1(\Gamma)$-norm and the $O(h^2)$ error estimate in the $L^2(\Gamma)$-norm on a smooth closed surface $\Gamma$. Here $h$ is the maximum side length among all cubes intersected by the recovered (discrete) surface. To access the local error, we introduce an error indicator based on elementwise residual and surface curvature. A grid adaptation strategy based on this indicator leads to optimal convergence of numerical solution with respect to the number of degrees of freedom.

In the paper, we consider regular Laplace-Beltrami type problems as well as singular-perturbed advection-diffusion
equations on surfaces. The latter case is of interest for a number of applications such as the transport of surfactants along fluidic interfaces or a pollutant transport in fractured media. For the advection dominated problem we consider a stabilized variant of the trace finite element method as well as layer fitted meshes.
The remainder of the paper is organized as follows.  Section~\ref{s_setup} sets up a problem.  
In section~\ref{s_FEM}, we introduce a  finite element method. Analysis of the finite element method  is presented in section~\ref{s_error}. It includes a well-posedeness result, a priori and a posteriori error estimates. Section~\ref{s_numer} collects the result of several numerical experiments. Finally, section~\ref{s_concl} contains some closing remarks.

\section{Problem formulation}\label{s_setup}

Let $\Omega$ be an open domain in $\mathbb{R}^3$ and $\Gamma$ be a connected $C^2$ compact
hyper-surface contained in $\Omega$.
For a sufficiently smooth function $g:\Omega\rightarrow \mathbb{R}$ the tangential derivative
on $\Gamma$ is defined by
\begin{equation}
 \nabla_{\Gamma} g=\nabla g-(\nabla g\cdot \bn)\bn,\label{e:2.1}
\end{equation}
where $\bn$ denotes the unit normal to $\Gamma$. Denote  by $\Div_\Gamma=\nabla_\Gamma\cdot$ the surface divergence operator and by $\Delta_{\Gamma}=\nabla_\Gamma\cdot\nabla_\Gamma$ the Laplace-Beltrami operator on $\Gamma$.
The simplest elliptic PDE on $\Gamma$ is the classical Laplace-Beltrami equation
\begin{equation}
  -\Delta_{\Gamma} u=f\quad\text{on}~~\Gamma,\quad\text{with}~\int_\Gamma f\,\ds=0. \label{eq:LB}
\end{equation}

Although, \eqref{eq:LB} is an interesting problem arising in many applications, we shall
consider a slightly more general problem on $\Gamma$. To motivate it, consider $\mathbf{w}:\Omega\rightarrow\mathbb{R}^3$, a given  velocity field in $\Omega$. If the surface  $\Gamma$ evolves with a normal velocity of $\bw\cdot\bn$ (e.g., $\Gamma$ is passively advected by the velocity field $\bw$), then the conservation of a scalar quantity $u$ with a diffusive flux on $\Gamma(t)$ leads to the surface PDE, see, e.g., \cite{DE2013}:
\begin{equation}
\dot{u} + (\Div_\Gamma\bw)u -\varepsilon\Delta_{\Gamma} u=0\quad\text{on}~~\Gamma(t),
\label{transport}
\end{equation}
where $\dot{u}$ denotes the advective material derivative, $\eps>0$ is the diffusion coefficient.
If we assume $\bw\cdot\bn=0$, i.e. the advection velocity is everywhere tangential to the surface, and the surface is steady in the geometric sense, then the surface advection-diffusion equation takes the form:
\begin{equation}
u_t+ \mathbf{w}\cdot\nabla_{\Gamma} u+ (\Div_\Gamma\bw)u -\varepsilon\Delta_{\Gamma} u=0
\quad\text{on}~~\Gamma.
\label{e:2.2}
\end{equation}
Although the methodology of the paper is applied to the parabolic equations \eqref{e:2.2},
we shall present the method and analysis  for the stationary problem:
\begin{equation}
  -\varepsilon\Delta_{\Gamma} u+\mathbf{w}\cdot\nabla_{\Gamma} u + (c+\Div_\Gamma\bw)\,u =f\quad\text{on}~~\Gamma, \label{problem}
\end{equation}
with  $f\in L^2(\Gamma)$ and $c=c(\bx)\in L^\infty(\Gamma)$.
If $\bw=0$ and $c=0$, we recover the classical Laplace-Beltrami problem \eqref{eq:LB}. Otherwise we assume $\bw\in H^{1,\infty}(\Gamma)$.

We need the following identity for integration by parts over $\Gamma$:
\begin{equation}\label{intByParts}
\int_\Gamma q (\Div_\Gamma \blf) + \blf\cdot\nabla_\Gamma q\,d\bs=\int_\Gamma \kappa (\blf\cdot\bn)q\,d\bs.
\end{equation}
for smooth $q$ and vector field $\blf$,  where $\kappa = \Div_\Gamma \bn$ is the (doubled) surface mean curvature.
Applying \eqref{intByParts} and $\mathbf{w}\cdot\bn=0$ one finds the identity
\[
\int_{\Gamma}(\bw\cdot\nat u) v\, \ds=-\int_{\Gamma}(\bw\cdot\nat v) u\,\ds
-\int_{\Gamma}(\Div_\Gamma\bw) \, uv\, \ds.
\]
Integrating \eqref{problem} over $\Gamma$ and applying the above identity with $v=1$, one finds that for $c=0$ the source term in \eqref{problem} should satisfy the zero mean constraint $\int_\Gamma f\,\ds=0$.

  Introduce the bilinear form and the functional:
\begin{equation*}
\begin{aligned}
a(u,v)&:=\int_{\Gamma}\eps\nabla_{\Gamma} u\cdot\nabla_{\Gamma} v  - (\bw\cdot\nat v) u
+c\, uv\, \ds, \\
f(v)&:=\int_{\Gamma}fv\, \ds.
\end{aligned}
\end{equation*}
The weak formulation of  \eqref{problem} is as follows: Find $u\in V$ such that
\begin{equation}
a(u,v)=f(v) \qquad \forall v\in V,\label{weak}
\end{equation}
with
$$V= \left\{
\begin{array}{ll}
 \{v\in H^1(\Gamma)\ |\ \int_{\Gamma} v\, \ds=0\} & \hbox{if } c+\Div_\Gamma\bw= 0,\\
H^1(\Gamma) & \hbox{ otherwise}.
     \end{array}
     \right.
$$
For functions satisfying zero integral mean condition, the following Poincare's  inequality holds:
\begin{equation}
\|v\|_{L^2(\Gamma)}^2 \le C_P \|\nabla_\Gamma v\|_{L^2(\Gamma)}^2\quad
\forall~v\in V,~\mbox{s.t.}~~ \int_{\Gamma} v\, \ds=0.
\label{FdrA}
\end{equation}

We shall assume
\begin{equation}\label{A1}
\begin{array}{ll}
\eps C_P^{-1}-\sup_{\bx\in\Gamma}|\Div_\Gamma\bw(\bx)|\ge c_0>0 & \hbox{if } c+\Div_\Gamma\bw= 0,\\[1ex]
c+\frac12\Div_\Gamma\bw\ge c_0>0~~\text{on}~~\Gamma& \mbox{otherwise}.
     \end{array}
 \end{equation}
If a time stepping scheme is used for \eqref{e:2.2} and $c$ is proportional to the reciprocal of the time step, then
the assumption is not restrictive.

The Lax-Milgram lemma and \eqref{FdrA} immediately yield the well-posedness result for \eqref{weak}. A higher
smoothness of the solution follows from a regularity  result for the Laplace-Beltrami equation in \cite{Aubin}.
\begin{proposition}
Assume \eqref{A1}, then  there exists a unique solution of \eqref{weak}, satisfying
\begin{equation}\label{apr_est1}
\frac{c_0}2\|u\|_{L^2(\Gamma)}^2+\eps\|\nabla_\Gamma u\|_{L^2(\Gamma)}^2\le 2c_0^{-1}\|f\|_{L^2(\Gamma)}^2
\end{equation}
and
\[
\|u\|_{H^2(\Gamma)}^2\le C_2\|f\|_{L^2(\Gamma)}^2,
\]
with a constant $C_2$ depending on $\eps$, $\bw$, $c_0$, and $\Gamma$.
\end{proposition}

\section{The trace finite element method} \label{s_FEM}
 In this section we review the trace FEM. The method  developed in this section is an extension  of the trace FEM  time introduced in \cite{ORG09}.

\subsection{The idea of the method} Assume we are given a polyhedral subdivision $\mathcal{T}_h$ of the bulk domain $\Omega$ and $V_h$ is a $H^1(\Omega)$ conforming finite element space. Consider all traces of functions from $V_h$ on $\Gamma$ and denote the
resulting space of surface functions by $V_h^\Gamma$. Now one substitutes $V$ by $V_h^\Gamma$ in the weak formulation \eqref{weak} to obtain a finite element formulation. If $\Gamma$ is given implicitly or no parametrization of $\Gamma$ is known, a problem of integration of finite element functions over $\Gamma$ arises. Hence in practice, $\Gamma$ in the finite element formulation is replaced by an approximate (``discrete'') surface $\Gamma_h$ such that integration over $\Gamma_h$ is feasible. For example, $\Gamma_h$ is piecewise polygonal.
Trace space for $V_h$ is now defined over $\Gamma_h$ rather than $\Gamma$ and problem data is extended from $\Gamma$ to $\Gamma_h$. Substituting $\Gamma$ by $\Gamma_h$ introduces a geometric error in the method that has to be quantified. It is remarkable,  however, that a suitable $\Gamma_h$ can be easily constructed for an implicitly given surface without any knowledge of parametrization. Moreover, in some applications $\Gamma$ is not known, and $\Gamma_h$ is recovered from a solution to an (discretized) equation. The trace finite element method is perfectly suited for such a situation.

It is clear from this general description that both $\Gamma$ and $\Gamma_h$ may cut the bulk mesh in an arbitrary way.
So the trace finite element method can be related to the family of unfitted finite elements, well developed  for equations posed in bulk domains, such as XFEM or cut-FEM. To build a basis or a frame in $V^\Gamma_h$, one may consider
traces of basis functions from $V_h$ on $\Gamma_h$.  It is also clear that only those  bulk
basis  functions are active that have their support intersected by $\Gamma_h$.

Further in the paper, we study the method if $\mathcal{T}_h$ is a cubic octree mesh, $V_h$ is a space of piecewise trilinear continuous finite elements, and $\Gamma_h$ is reconstructed by a variant of marching cubes method from a piecewise trilinear interpolant to a level set function of $\Gamma$.

\subsection{The method}
Consider an octree cubic mesh $\mathcal{T}_h$ covering the bulk domain $\Omega$. We assume that the mesh
is gradely refined, i.e. the sizes of two neighbouring cubes differ at most by a factor of 2. Such octree grids are also known as balanced. The method applies for unbalanced octrees, but in our analysis and experiments we use balanced grids.   

Denoted by $\Gamma_h$ a given polygonal approximation of $\Gamma$. We assume that
$\Gamma_h$ is a $C^{0,1}$ surface without  boundary and $\Gamma_h$ can be partitioned in planar triangular segments:
\begin{equation} \label{defgammah}
 \Gamma_h=\bigcup\limits_{T\in\mathcal{F}_h} T,
\end{equation}
where $\mathcal{F}_h$ is the set of all  triangular segments $T$.
Without loss of generality we assume that for any $T\in\mathcal{F}_h$ there is only \textit{one}
cube $S_T\in\mathcal{T}_h$ such that $T\subset S_T$ (if $T$ lies on a side shared by two cubes, any of these
two cubes can be chosen as $S_T$).

In practice, we construct $\Gamma_h$ as follows. Let $\phi_h$ be a  piecewise trilinear continuous function with respect
to the octree grid $\mathcal{T}_h$ and consider its zero level set
\[
\widetilde{\Gamma}_h:=\{\bx\in\Omega\,:\, \phi_h(\bx)=0 \}.
\]
We assume that $\widetilde{\Gamma}_h$ is an approximation to $\Gamma$. This  is  a reasonable assumption if $\phi_h$ is
an interpolant to $\phi$, a level set function of $\Gamma$; one example of $\phi$ is the signed distance function for $\Gamma$. To define $\phi_h$ one only should be able to prescribe in each node an approximate distance to $\Gamma$. Alternatively, in some applications, $\phi_h$ is recovered from a solution of a discrete indicator function equation (e.g. in the level set or the volume of fluid methods), without any direct knowledge of $\Gamma$.

Once $\phi_h$ is known, we recover $\Gamma_h$ by the cubical marching squares method from \cite{MCM2} (a variant of the
very well-known marching cubes method). The method provides a triangulation of $\widetilde{\Gamma}_h$ within each
cube such that the global triangulation is continuous, the number of triangles within each cube is finite and bounded
by a constant independent of $\widetilde{\Gamma}_h$ and a number of refinement levels. Moreover, the vertices of triangles from $\mathcal{F}_h$ are lying on $\widetilde{\Gamma}_h$.

An illustration of a bulk mesh and a surface triangulation is given in Figure~\ref{fig:tri}. The mesh shown in this figure was obtained by representing a torus $\Gamma$   implicitly by its signed distance function, constructing the piecewise trilinear continuous interpolant of this distance function  and then applying the cubical marching squares method for reconstructing $\Gamma_h$ from the zero level of this interpolant.
\begin{figure}[ht!]
\begin{center}
\begin{picture}(420,160)(0,0)
\put(0,0){\includegraphics[trim=10 28 10 10,clip,width=0.45\textwidth]{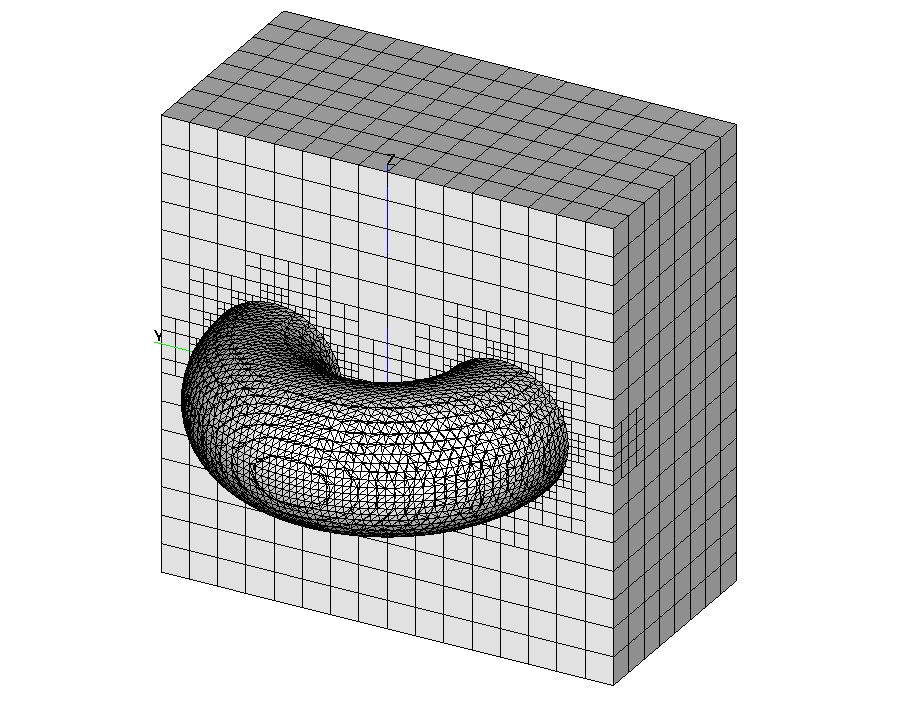}} 
\put(230,20){\includegraphics[width=0.33\textwidth]{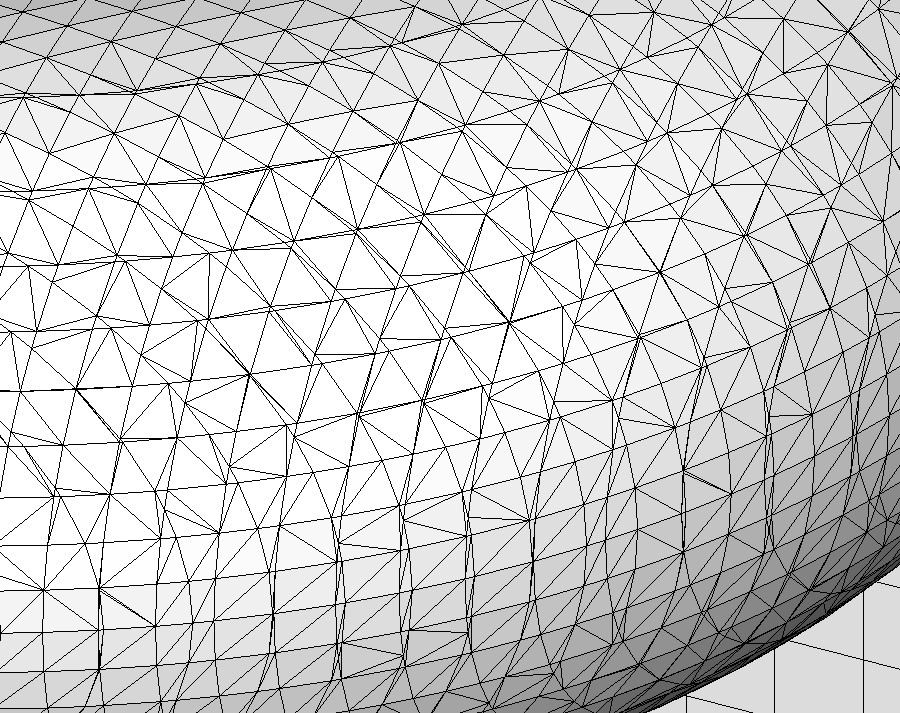}}
\end{picture}
\caption{Left: A cutaway of a bulk domain shows the bulk octree mesh three times refined towards the surface and the resulting   approximate surface $\Gamma_h$ for the part of a torus. Right: The zoom-in of the resulting surface triangulation. The triangulation does not satisfy a  minimal angle condition.}
\label{fig:tri}
\end{center}
\end{figure}

Note that the resulting triangulation $\mathcal{F}_h$  is \emph{not} necessarily  regular, i.e. elements from $T$ may have
very small internal angles and the size of neighboring triangles can vary strongly, cf.~Figure~\ref{fig:tri} (right).
Thus, $\Gamma_h$ is not a ``triangulation of $\Gamma$'' in the  usual sense (an $O(h^2)$ approximation of $\Gamma$, consisting of \textit{regular} triangles).
\smallskip

The surface finite element space is \textit{the space of traces on $\Gamma_h$ of all piecewise trilinear continuous functions with respect to the outer triangulation $\mathcal{T}_h$}.
This can be formally defined as follows.

Consider now the volumetric finite element space of all piecewise trilinear continuous functions with respect to the bulk octree mesh $\mathcal{T}_h$:
\begin{equation}
 V_h:=\{v_h\in C(\Omega)\ |\ v|_{S}\in Q_1~~ \forall\ S \in\mathcal{T}_h\},\quad\text{with}~~
 Q_1=\mbox{span}\{1,x_1,x_2,x_3,x_1x_2,x_1x_3,x_2x_3,x_1x_2x_3\}.
 \label{e:2.6}
\end{equation}
$V_h$ induces the following space on $\Gamma_h$:
\begin{equation}
 V_h^{\Gamma}:=\{\psi_h\in H^1(\Gamma_h)\ |\ \exists ~ v_h\in V_h\  \text{such that }\ \psi_h=v_h|_{\Gamma_h}\}.
\label{e:fem-space}
\end{equation}

Given the surface finite element space $V_h^{\Gamma}$,  the finite element
discretization of \eqref{problem} is as follows:   Find $u_h\in V_h^{\Gamma}$ such that
\begin{equation}
\eps\int_{\Gamma_h}\nath u_h\cdot\nath v_h\, - (\bw_h\cdot\nath v_h) u_h
+c_h \, u_hv_h\, \ds_h =\int_{\Gamma_h}f_h v_h\, \ds_h \label{FEM}
\end{equation}
for all $v_h\in V_h^{\Gamma}$. Here $\mathbf{w}_h$, $c_h$ and $f_h$ are some approximations of the problem data on $\Gamma_h$.
A well-posedness result for \eqref{FEM} will be proved in the next section.

\subsection{Variants of the method}
Here we discuss several extensions of the trace finite element method \eqref{FEM}, which can be advantageous
in some situations. One obvious update of the method is the following one. Define a subdomain $\omega_h$ of $\Omega$
consisting only of those end-level cubic cells that contain $\Gamma_h$:
\begin{equation} \label{defomeg}
 \omega_h= \bigcup_{T \in \mathcal{F}_h} S_T.
\end{equation}
For the outer finite element space, instead of piecewise trilinear continuous functions in $\Omega$, consider all such
function only in $\omega_h$:
\[
 V_h^{\omega}:=\{v_h\in C(\omega_h)\ |\ v|_{S}\in Q_1~~ \forall\ S \in\mathcal{T}_h\}.
\]
Further, define the space of traces of functions from $V_h^{\omega}$ on $\Gamma_h$:
\[
 V_h^{\omega,\Gamma}:=\{\psi_h\in H^1(\Gamma_h)\ |\ \exists ~ v_h\in V_h^{\omega}\  \text{such that }\ \psi_h=v_h|_{\Gamma_h}\}.
\]
It is clear that $V_h^{\Gamma}\subset V_h^{\omega,\Gamma}$. When grid is locally refined, the  dimension of the space $V_h^{\omega,\Gamma}$ can be larger for the following reason: The inter-element  continuity of bulk finite element functions
imposes algebraic constraints in hanging nodes. For the space $V_h^{\omega,\Gamma}$ such constraints should be imposed only for hanging nodes shared by two cubic cells from $\omega_h$, but not for hanging nodes lying on the boundary of $\omega_h$,
which are now available for extra degrees of freedom.

We performed numerical experiments with the trace finite element space from \eqref{e:fem-space} and observed optimal
convergence behavior with respect to the number of degrees of freedom (cf. Section~\ref{s_numer}).
We expect that the method with $V_h^{\omega,\Gamma}$ instead of $V_h^{\Gamma}$ behaves similarly.

Furthermore, one may relax the continuity assumption for the bulk finite element functions in $\omega_h$ and consider a
 discontinuous Galerkin finite element method. This is an interesting (and natural in some sense) alternative for an octree based finite element method, which we plan to study elsewhere.

As usual with transport-diffusion problems, the advection terms can written in several equivalent ways, leading, however, to different discretizations. For example, instead of $- (\bw_h\cdot\nath v_h) u_h$  one way choose the
`skew-symmetric' variant
\[
\frac12\int_{\Gamma_h}(\bw_h\cdot\nath u_h) v_h\, - (\bw_h\cdot\nath v_h) u_h+ (\Div_{\Gamma_h}\bw_h) \, u_hv_h\, \ds_h.
\]
We shall analyse the `conservative' formulation \eqref{FEM}, since it avoids the surface divergence of $\bw_h$.
Our motivation is that for a polygonal  surface reconstruction the term $\Div_{\Gamma_h}\bw_h$ leads to the consistency error of the method of
order $O(h)$, if $\bw_h$ is a smooth extension of $\bw$ from $\Gamma$. For the `conservative' form we a able to prove $O(h^2)$ accuracy of the method.

 Below we discuss a few more developments of the trace finite element
method known from the literature.

\subsubsection{Full gradient method} The ``full gradient'' variant of the trace finite element method was suggested
in \cite{Alg2} and studied in \cite{Alg2,Reusken2014}. The modification is aimed on improving algebraic properties
of the stiffness matrix of the method. The rationality behind the full gradient method is clear from the following observation.
For solution $u$ of \eqref{problem}, denote by $u^e$ its normal extension to an arbitrary small neighbourhood $\mathcal{O}(\Gamma)$ of $\Gamma$, i.e., $u^e$ is constant along normal directions to $\Gamma$. Note that $\nabla_{\Gamma} u= \nabla u^e$ and so $u^e$ satisfies the identity \eqref{problem} with surface gradients (in the diffusion term) replaced
by full gradients:
\[
\int_{\Gamma}\eps\nabla u^e\cdot\nabla v\, - (\bw\cdot\nabla_{\Gamma} v) u^e
+c \, u^ev\, \ds =\int_{\Gamma}f v\, \ds
\]
for all $v$ sufficiently smooth in $\mathcal{O}(\Gamma)$ ($v$ is not necessarily constant along normals).
This identity shows that the following finite element formulation is consistent: Find $u_h\in V_h$ satisfying
\begin{equation}
\int_{\Gamma_h}\eps\nabla u_h\cdot\nabla v_h\, - (\bw_h\cdot\nath v_h) u_h
+c_h \, u_hv_h\, \ds_h =\int_{\Gamma_h}f_h v_h\, \ds_h \label{FEM_fg}
\end{equation}
for all $v_h\in V_h$.

The full-gradient method \eqref{FEM_fg} uses the bulk finite element space $V_h$ instead of the surface finite element space $V_h^\Gamma$ in \eqref{FEM}. However, practical implementation of both methods uses the frame of all bulk finite element nodal basis functions $\phi_i\in V_h$ such that  $\mbox{supp}(\phi_i)\cap\Gamma_h\neq\emptyset$. Hence the active degrees of freedom  in both methods are the same. The stiffness matrices are, however, different.
For the case of the Laplace-Beltrami problem and a  regular quasi-uniform tetrahedral grid, results in \cite{Alg2,Reusken2014}  show that the conditioning of the (diagonally scaled) stiffness matrix of the full gradient method substantially improves over the conditioning of the matrix for \eqref{FEM}, for the expense of a slight deterioration
of the accuracy of the method.

\subsubsection{SUPG stabilized method}

Similar to the plain Galerkin finite element for  advection-diffusion equations the method \eqref{FEM} is prone to instability unless mesh is sufficiently fine such that the mesh Peclet number is less than one.

In \cite{ORXimanum}, a SUPG type stabilized trace finite element method was introduced  and analysed for $P_1$ continuous  bulk finite elements.  The stabilized formulation reads:
 Find $u_h\in V_h^{\Gamma}$ such that
\begin{multline}
 \int_{\Gamma_h}\eps\nath u_h\cdot\nath v_h\, - (\bw_h\cdot\nath v_h) u_h+c_h\, u_hv_h\, \ds_h \\
  +\sum_{T\in\mathcal{F}_h}\delta_T\int_{T}(-\eps\Delta_{\Gamma_h}u + \bw_h\cdot\nath u + (c_h+\Div_{\Gamma_h}\bw_h) \, u)\,\bw_h\cdot\nath v\, \ds_h\\ =\int_{\Gamma_h}f_h v\, \ds_h + \sum_{T\in\mathcal{F}_h}\delta_T\int_{T}f_h(\bw_h\cdot\nath v)\, \ds_h\quad \forall~ v_h\in V_h^{\Gamma}. \label{FEM_SUPG}
\end{multline}
The stabilization parameter  $\delta_T$ depends on $T \subset S_T$. The side length of the cubic cell $S_T$ is denoted by $h_{S_T}$.  Let $\displaystyle \mathsf{Pe}_T:=\frac{h_{S_T} \|\mathbf{w}_h\|_{L^\infty(T)}}{2\eps}$
be the cell Peclet number.
We take
\begin{equation}
 \widetilde{\delta_T}=
\left\{
\begin{aligned}
&\frac{\delta_0 h_{S_T}}{\|\mathbf{w}_h\|_{L^\infty(T)}} &&\quad \hbox{ if } \mathsf{Pe}_T> 1,\\
&\frac{\delta_1 h^2_{S_T}}{\eps}  &&\quad \hbox{ if } \mathsf{Pe}_T\leq 1,
\end{aligned}
\right.\quad\text{and} \quad\delta_T=\min\{\widetilde{\delta_T},c^{-1}\}, \label{e:2.10}
\end{equation}
with some given positive constants  $\delta_0,\delta_1\geq 0$.

\subsubsection{Gradient-jump stabilized method}

Another way of improving algebraic properties of the trace finite element method was suggested in \cite{Alg1}.
In that paper, the authors introduced the term
\[
J(u,v)=\sum_{F\in\omega_h^F}\sigma_F \int_{\partial F}\llbracket\bn_F\cdot\nabla u\rrbracket\llbracket\bn_F\cdot\nabla v\rrbracket,
\]
with
$
\sigma_E =0(1).
$ Here $\omega_h^F$ denotes the set of all internal interfaces between cubic cells in $\omega_h$, $\bn_F$ is the normal vector
for interface $F$, and $\llbracket\bn_F\cdot\nabla u\rrbracket$ denotes the jump of the normal derivative of $u$ across $F$.
Now, the edge-stabilized trace finite element reads: Find $u_h\in V_h$ such that
\begin{equation}
\eps\int_{\Gamma_h}\nath u_h\cdot\nath v_h\, - (\bw_h\cdot\nath v_h) u_h+c_h\, u_hv_h\, \ds_h + J(u_h,v_h) =\int_{\Gamma_h}f_h v_h\, \ds_h \label{FEM_edge}
\end{equation}
for all $v_h\in V_h$.

For $P_1$ continuous bulk finite element methods on quasi-uniform regular tetrahedal meshes and the Laplace-Beltrami equation, the paper \cite{Alg1} shows the optimal orders of convergence of the trace finite element method \eqref{FEM_edge}
and improved algebraic properties of the stiffness matrix.
\smallskip

In this paper, we consider the trace finite element in \eqref{FEM}, its full-gradient variant in \eqref{FEM_fg},
and the SUPG stabilized version \eqref{FEM_SUPG} for the case of convection-dominated equations. We are not studying the edge stabilized version here.

\section{Well-posedness and error analysis of the trace FEM}\label{s_error}
In this section, we state a well-posedness result for \eqref{FEM}. Further, an error analysis of the trace method is presented for a regular problem (we assume $\ep$ is not too small). In particular, we shall prove optimal order of convergence of the method in $H^1$ and $L^2$ surface norms.
 The analysis can be extended to the full-gradient and the  SUPG stabilized  formulations following the arguments of \cite{Reusken2014} and \cite{ORXimanum}, respectively. However, in this paper we restrict ourselves to the
original formulation \eqref{FEM}.

Let $\{\mathcal{T}_h\}_{h>0}$ be a family of octree meshes covering
the domain $\Omega$. Parameter $h$ denotes the maximum size of a cubic cell from $\mathcal{T}_h$.
For the sake of analysis, we assume that the maximum number of local refinement levels is bounded from above independently of $h$.
Further, we need some more notations and definitions.

\subsection{Preliminaries}

For a given octree mesh $\mathcal{T}_h$ let $\Psi_h$ be the set of  all nodal basis functions of the bulk  finite element space $V_h$. Here and further denote by $\Omega_h$ the union of all supports of nodal basis functions intersected by $\Gamma_h$:
\[
\Omega_h=\bigcup_{\psi\in\Psi_h}\{\mbox{supp}(\psi)\,|\,\mbox{supp}(\psi)\cap\Gamma_h\neq\emptyset\}.
\]
For the surface $\Gamma$, we define its $h$-neighborhood:
\begin{equation}
 U_h:=\{\mathbf{x}\in \mathbb{R}^3\ |\ \mathrm{dist}(\mathbf{x},\Gamma)< \tilde{c}\, h\},\label{e:3.1}
\end{equation}
and assume  that $\tilde{c}$ is sufficiently large and $h$ is sufficiently  small such that $\Omega_h\subset U_h\subset\Omega$ and
\begin{equation}
 5\tilde{c}\,h<\left(\mathrm{max}_{i=1,2} \|\kappa_i\|_{L^{\infty}(\Gamma)}\right)^{-1}\label{e:3.2}
\end{equation}
holds, with $\kappa_i$ being the principal curvatures of $\Gamma$.

Let $d : U_h\rightarrow \mathbb{R}$ be the signed distance function.
Thus $\Gamma$ is the zero level set of $d$ and $\Gamma\in C^2~\Rightarrow~d\in C^2(U_h)$. We assume $d<0$ in the interior of $\Gamma$ and $d>0$ in the exterior and define $\mathbf{n}(\mathbf{x}):=\nabla d(\mathbf{x})$ for all $\mathbf{x}\in U_h$. Hence,  $\mathbf{n}$ is the normal vector on $\Gamma$ and $|\mathbf{n}(\bx)|=1$ for all $\bx\in U_h$.  The Hessian of $d$ is denoted by
\begin{equation*}
 \mathbf{H}(\bx):= \nabla^2 d(\bx)\in \mathbb{R}^{3\times 3},\quad \bx\in U_h.
\end{equation*}
The eigenvalues of $\mathbf{H}(\bx)$ are denoted by $\kappa_1(\bx)$, $\kappa_2(\bx)$, and~$0$.
 For $\bx\in\Gamma$ the eigenvalues $\kappa_i, i=1,2,$ are the principal curvatures.
For each $\bx\in U_h$,  define the projection $\bp: U_h\rightarrow\Gamma$ by
\begin{equation*}
 \bp(\bx)=\bx-d(\bx)\bn(\bx).
\end{equation*}
Due to the assumption \eqref{e:3.2},  the decomposition $\bx=\bp(\bx)+d(\bx)\bn(\bx)$ is unique.
We need the orthogonal projector
\begin{equation*}
 \bP(\bx):=\mathbf{I}-\bn(\bx)\bn(\bx)^T, \quad \hbox{for }\bx\in U_h.
\end{equation*}
The tangential derivative can be written as $\nabla_{\Gamma} g(\bx)=\bP\nabla g(\bx)$ for $\bx\in \Gamma$.
One can verify that for this projection and for the Hessian $\mathbf H$ the relation
$\mathbf H\bP=\bP\mathbf{H}=\mathbf{H}$ holds.
Similarly, define
\begin{equation*}
 \bP_h(\bx):=\mathbf{I}-\bn_{h}(\bx)\bn_{h}(\bx)^T, \quad \hbox{for }\bx\in \Gamma_h,~\bx \hbox{ is not on an edge},
\end{equation*}
where $\bn_{h}$ is the unit (outward pointing) normal at $\bx\in \Gamma_h$. The tangential
derivative along $\Gamma_h$ is given by $\nabla_{\Gamma_h} g(\bx)=\bP_h(\bx)\nabla g(\bx)$.

Assume the following estimates on how well $\Gamma_h$ approximates $\Gamma$:
\begin{align}
&\mathrm{ess\ sup}_{\bx\in \Gamma_h}|d(\bx)| \le c_1 h^2, \label{e:3.9}\\
&\mathrm{ess\ sup}_{\bx\in \Gamma_h}|\bn(\bx)-\bn_{h}(\bx)| \le c_2   h,\label{e:3.10}
\end{align}
with constants $c_1$, $c_2$ are independent of $h$.
The assumption is reasonable if $\Gamma$ is defined as the zero level of a (locally) smooth level set function $\phi$
and $\Gamma_h$ is reconstructed as described in the previous  section from the zero level set of $\phi_h\in V_h$,
where $\phi_h$ interpolates $\phi$ and it holds
\[
\|\phi-\phi_h\|_{L^\infty(U_h)}+h\|\nabla(\phi-\phi_h)\|_{L^\infty(U_h)}\le c\,h^2.
\]

In the remainder, $A\lesssim B $ means $A\leq c\, B $ for some positive constant $c$ independent of $h$.

For $\bx\in\Gamma_h$, define
 $\mu_h(\bx) = (1-d(\bx)\kappa_1(\bx))(1-d(\bx)\kappa_2(\bx))\bn^T(\bx)\bn_h(\bx)$.
The surface measures  $\ds$ and $\ds_{h}$ on $\Gamma$ and $\Gamma_h$, respectively, are related by
\begin{equation}
 \mu_h(\bx)\ds_h(\bx)=\ds(\bp(\bx)),\quad \bx\in\Gamma_h. \label{e:3.16}
\end{equation}
The assumptions \eqref{e:3.9} and \eqref{e:3.10} imply that
\begin{equation}
 \mathrm{ess\ sup}_{\bx\in\Gamma_h}(1-\mu_h)\lesssim h^2,\label{e:3.17}
\end{equation}
cf. (3.37) in \cite{ORG09}.
The solution of  \eqref{problem}  and its data are defined on $\Gamma$,
while the finite element method is defined on $\Gamma_h$.
Hence, we need a suitable extension of a function from $\Gamma$ to its neighborhood.
For a function $v$ on $\Gamma$ we define
\begin{equation*}
 v^e(\bx):= v(\bp(\bx)) \quad \hbox{for all } \bx\in U_h.
\end{equation*}
The following formula for this lifting function are known (cf. section 2.3 in \cite{DD07}):
\begin{align}
 \nabla u^e(\mathbf{x}) &= (\mathbf{I}-d(\mathbf{x})\mathbf{H})\nabla_{\Gamma} u(\bp(\bx)) \quad \hbox{ in } U_h,\label{grad1}\\
 \nabla_{\Gamma_h} u^e(\bx) &= \bP_h(\bx)(\mathbf{I}-d(\mathbf{x})\mathbf{H})\nabla_{\Gamma} u(\bp(\bx)) \quad \hbox{ a.e. on } \Gamma_h,\label{grad2}
\end{align}
with $\mathbf{H}=\mathbf{H}(\mathbf{x})$.
By direct computation one derives the relation
\begin{multline}
\nabla^2 u^e(\bx)=(\bP -d(\bx) \mathbf{H})\nabla^2 u(\bp(\bx))(\bP -d(\bx) \mathbf{H})-(n^T\nabla u(\bp(\bx))\mathbf{H}\\
-(\mathbf H\nabla u(\bp(\bx))) \bn^T-\bn (\mathbf H\nabla u(\bp(\bx)))^T -d\nabla \mathbf H:\nabla u(\bp(\bx)).
 \label{e:3.14n}
\end{multline}
For  sufficiently smooth $u$ and $|\mu| \leq 2$, using this relation  one obtains the estimate
 \begin{equation}
  |D^{\mu} u^e(\bx)|\lesssim \left(\sum_{|\mu|=2}|D_{\Gamma}^{\mu} u(\bp(\bx))| + |\nabla_{\Gamma} u(\bp(\bx))|\right)\quad \hbox{in } U_h,
\label{e:3.13}
 \end{equation}
(cf. Lemma 3 in \cite{Dziuk88}). This further leads to (cf. Lemma 3.2 in \cite{ORG09}):
 \begin{equation}
 \|D^{\mu} u^e\|_{L^2(U_h)}\lesssim\sqrt{h}\|u\|_{H^2(\Gamma)}, \quad |\mu|\leq 2. \label{e:3.14}
 \end{equation}
 \smallskip

 For the analysis, we shall assume that $\mathbf{w}_h$, $c_h$ and $f_h$ are defined as the extensions  of $\bw$, $c$, and $f$, respectively, along normals to $\Gamma$, i.e. $\mathbf{w}_h=\mathbf{w}^e$, $c_h=c^e$ and $f_h=f^e$ on $\Gamma_h$.

\subsection{Well-posedness}
We need some further notations. For each $T\in\mathcal{F}_h$, denote by ${\mathbf{m}_h}|_{E}$ the outer normal to an edge $E$ in the plane which contains element $T$. For a curved geometries, ``tangential'' normal vectors to $E$ from two  different sides are not necessarily collinear, cf. Figure \ref{fig:3.1}.
\begin{figure}[ht!]
\begin{center}
  \includegraphics[trim=100 100 100 50,clip,width=0.45\textwidth]{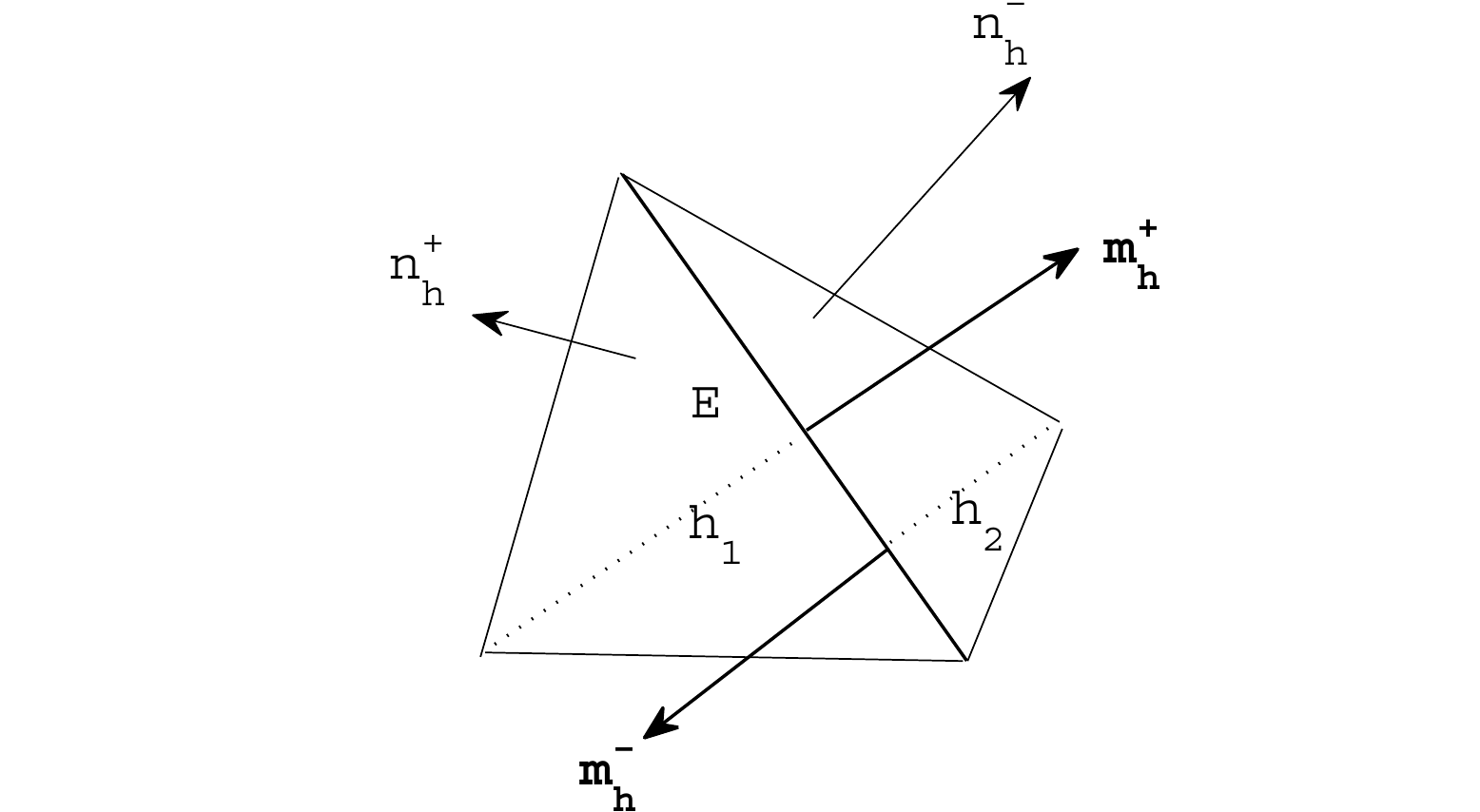}
    \caption{Two neighboring surface triangles may belong to two different planes and so two tangential
    normals $\mathbf{m}_h^+$ and $\mathbf{m}_h^-$ to the edge $E$ are not necessarily collinear.     \label{fig:3.1}}
\end{center}
 \end{figure}
 Let $\llbracket\mathbf{m}_h\rrbracket|_{E}=\mathbf{m}_h^+  +\mathbf{m}_h^-$ be the jump of two outward normals to the edge $E$. For a `usual' planar case, this jump is zero. Over $\Gamma_h$, these jumps produce additional consistency
 term in the integration by parts  formula:
\begin{multline}\label{inth}
-\int_{\Gamma_h}(\bw_h\cdot\nabla_{\Gamma_h} v_h) u_h\,\ds_h = \frac12\int_{\Gamma_h}(\bw_h\cdot\nabla_{\Gamma_h} u_h) v_h-(\bw_h\cdot\nabla_{\Gamma} v_h) u_h+\Div_\Gamma\bw_hu_hv_h\, \ds_h\\ - \frac14\sum_{T\in\F_h}\int_{\partial T}( \bw_h\cdot\llbracket\mathbf{m}_h\rrbracket)u_hv_h\d\br.
\end{multline}
We introduce the notation for the finite element bilinear form and the functional:
\[
a_h(u,v):=\eps\int_{\Gamma_h}\nath u\cdot\nath v\, - (\bw_h\cdot\nath v) u
+c_h \, uv\, \ds_h,\quad f_h(v):=\int_{\Gamma_h}f_h v\,\ds_h.
\]

The following result is now straightforward.

\begin{proposition}
If one assumes
\begin{equation}\label{A1h}
\int_{\Gamma_h}(c_h+\frac12\Div_{\Gamma_h}\bw_h)v_h^2\,\ds_h- \frac14\sum_{T\in\F_h}\int_{\partial T}( \bw_h\cdot\llbracket\mathbf{m}_h\rrbracket)v_h^2\d\br\ge c_0\|v_h\|^2_{L^2(\Gamma_h)},~~\text{for all}~v_h\in V_h^\Gamma,
 \end{equation}
then the bilinear form is positive-definite on $V_h^\Gamma$ and so the trace finite element method  \eqref{FEM} is well-posed. The solution $u_h$ satisfies the a priori estimate,
\begin{equation}\label{apr_estFE}
\frac{c_0}2\|u_h\|_{L^2(\Gamma_h)}^2+\eps\|\nabla_{\Gamma_h} u\|_{L^2(\Gamma_h)}^2\le c_0^{-1}\|f\|_{L^2(\Gamma_h)}^2.
\end{equation}
\end{proposition}

If $c_h=0$ and $\bw_h=0$ on $\Gamma_h$ (the Laplace-Beltrami problem), then one imposes the zero mean conditions for the
solution and the right-hand side $\int_{\Gamma_h} f_h=\int_{\Gamma_h} u_h=0$ and proves the well-posedness result with the help of the Poincare inequality on $\Gamma_h$. For the mesh  independent stability bound, the Poincare constant should be uniformly bounded independent of $h$.
\medskip

To assess how restrictive the assumption in \eqref{A1h}, we invoke the following estimates (Lemmas~3.5 and~3.6 in \cite{ORXimanum}) to bound the edge terms:
\begin{equation*} \label{jump_est}
 |\bP(\bx)\llbracket\mathbf{m}_h\rrbracket(\bx)|\lesssim h^2
\qquad \text{for}~~\bx\in E, 
\end{equation*}
and
\begin{equation*} \label{edge_est}
 \sum_{T\in\F_h}\int_{\partial T}v_h^2 \, \ds_h  \lesssim  h^{-1} \|v_h\|^2_{L^2(\Gamma_h)}+
  h \|\nabla_{\Gamma_h} v_h\|^2_{L^2(\Gamma_h)} \quad \text{for all}~~v_h \in V_h^\Gamma.
\end{equation*}
We remark that in~\cite{ORXimanum} these estimates were proved for tetrahedra bulk subdivision, but all arguments
carry over to cubic bulk meshes once we assume that the number of refinement level in octree is bounded.
Since $\bP\bw^e=\bw^e$, we get with $\bw_h=\bw^e$ on $\Gamma_h$:
\[
\left|\sum_{T\in\F_h}\int_{\partial T}( \bw_h\cdot\llbracket\mathbf{m}_h\rrbracket)v_h^2\d\br\right|
\lesssim  h\|v_h\|_{L^2(\Gamma_h)}^2+h^3\|\nabla_{\Gamma_h} v_h\|_{L^2(\Gamma_h)}^2.
\]
Therefore, we conclude that for sufficiently fine mesh, the bilinear form $a_h$ is positive definite on $V_h^\Gamma$ and so the discrete problem is well-posed under condition similar to
the differential case: $c_h+\frac12\Div_{\Gamma_h}\bw_h\ge c_0>0$ on $\Gamma_h$.

\subsection{Consistency estimate} \label{sectconsistency}
The consistency error of the finite element method \eqref{FEM} is due to geometric errors resulting
from the approximation of $\Gamma$ by $\Gamma_h$. To estimate this geometric errors, we need  a few additional
definitions and results, which can be found in, for example, \cite{DD07}. For $\bx\in \Gamma_h$ define $\tilde{\bP}_{h}(\bx)= \mathbf{I}-\bn_h(\bx)\bn(\bx)^T/(\bn_h(\bx)\cdot\bn(\bx))$.
 One can represent the surface gradient of $u\in H^1(\Gamma)$ in terms of $\nath u^e$ as follows
\begin{equation*} \label{hhl}
 \nat u(\bp(\bx))=(\mathbf{I}-d(\bx)\mathbf{H}(\bx))^{-1} \tilde{\bP}_{h}(\bx) \nath u^e(\bx)~~ \hbox{ a.e. }\bx\in \Gamma_h.
\end{equation*}
Due to \eqref{e:3.16}, we get
\begin{equation*}\label{aux13}
 \int_{\Gamma}\nat u\nat v\, \ds=\int_{\Gamma_h} \mathbf{A}_h\nath u^e\nath v^e \, \ds_h \quad \hbox{for all } v\in H^1(\Gamma),
\end{equation*}
with $\mathbf{A}_h(\bx)=\mu_h(\bx) \tilde{\bP}^T_h(\bx)(\mathbf{I}-d(\bx)\mathbf{H}(\bx))^{-2}\tilde{\bP}_h(\bx)$.
 From $\bw \cdot \bn =0$ on $\Gamma$ and $\bw^e(\bx)=\bw(\bp(\bx))$, $\bn(\bx) =\bn(\bp(\bx))$ it follows that $\bn(\bx) \cdot \bw^e(\bx)=0$ and thus $\bw(\bp(\bx))= \tilde \bP_h(\bx)\mathbf{w}^e(\bx)$ holds. Using this, we get the relation
\begin{equation*}\label{aux14}
 \int_{\Gamma}(\bw\cdot \nat u) v \, \ds =  \int_{\Gamma_h} (\mathbf{B}_h \mathbf{w}^e \cdot \nath u^e)v^e\, \ds_h,
\end{equation*}
with $\mathbf{B}_h=\mu_h(\bx)\tilde{\bP}_h^T(I-d\mathbf{H})^{-1}\tilde\bP_h$.
We shall use the lifting procedure $\Gamma_h \to \Gamma$  given by
\begin{equation*} \label{deflift}
 v_h^l(\bp(\bx)): =v_h(\bx) \quad \text{for}~~ \bx\in\Gamma_h.
\end{equation*}
It is easy to see that $v_h^l\in H^1(\Gamma)$.


The following lemma estimates the consistency error of the finite element method \eqref{FEM}.

\begin{lemma} \label{thm3}
 Let $u\in H^2(\Gamma)$ be the solution of \eqref{problem}, then we have
\begin{equation*}
 \sup_{v_h\in V_h^{\Gamma}}\frac{|f_h(v_h)-a_h(u^e,v_h)|}{\|v_h\|_{H^1(\Gamma_h)}} \lesssim h^2(\|u\|_{H^1(\Gamma)}+\|f\|_{L^2(\Gamma)}).
\end{equation*}
\end{lemma}
\begin{proof}
The residual is decomposed as
\begin{equation} \label{er}
 f_h(v_h)-a_h(u^e,v_h)=f_h(v_h)-f(v_h^l)+a(u,v_h^l)-a_h(u^e,v_h).
\end{equation}
The following holds:
\begin{align*}
f(v^l_h)=&\,\int_{\Gamma}fv_h^l\, \ds= \int_{\Gamma_h}\mu_h f^e v_h\, \ds_h,\\
a(u,v^l_h)
=&\,\eps \int_{\Gamma}\nabla_{\Gamma} u \nabla_{\Gamma} v_h^l\, \ds
-\int_{\Gamma}(\bw\cdot\nabla_{\Gamma} v_h^l) u\,\ds +\int_{\Gamma}c\,uv_h^l\, \ds\\
=&\, \eps \int_{\Gamma_h} \mathbf{A}_h\nabla_{\Gamma_h} u^e\nabla_{\Gamma_h} v_h\, \ds_h
-\int_{\Gamma_h}(\mathbf{B}_h \bw^e\cdot\nath v_h ) u^e\, \ds_h+\int_{\Gamma_h}\mu_h c^eu^ev_h\, \ds_h.
\end{align*}
Substituting these relations into \eqref{er}   results in
\begin{multline}\label{aux12}
  f_h(v_h)-a_h(u^e,v_h)=
  \int_{\Gamma_h}(1-\mu_h)f^ev_h\, \ds_h+\eps\int_{\Gamma_h}(\mathbf{A}_h-\bP_h)\nabla_{\Gamma_h} u^e \cdot \nabla_{\Gamma_h}v_h\, \ds_h\\
-\int_{\Gamma_h}((\mathbf{B}_h-\bP_h)\mathbf{w}^e\cdot\nath v_h)  u^e\, \ds_h  + \int_{\Gamma_h}(\mu_h-1)c^eu^e v_h\, \ds_h =:I_1+I_2+I_3+I_4.
\end{multline}
We estimate the $I_i$ terms separately.
Applying  \eqref{e:3.17} we obtain
\begin{equation}
I_1+I_4\lesssim h^2 \|f^e\|_{L^2(\Gamma_h)}\|v_h\|_{L^2(\Gamma_h)}\label{aux15}.
\end{equation}
One can show, cf. (3.43) in \cite{ORG09}, the bound
\begin{equation*}
 |\bP_h-\mathbf{A}_h|=|\bP_h(\mathbf I-\mathbf A_h)|\lesssim h^2.
\end{equation*}
Using this and \eqref{grad2}  we obtain
\begin{equation}
I_2\lesssim h^2 \|\nath u^e\|_{L^2(\Gamma_h)}\|\nath v_h\|_{L^2(\Gamma_h)} \lesssim h^2 \|u^e\|_{H^1(\Gamma)}\|v_h\|_{H^1(\Gamma_h)}.
\end{equation}
Since $(\mathbf I-d\mathbf{H})^{-1}=\mathbf I+O(h^2)$, we also estimate
\begin{equation*}
 |\mathbf B_h-\bP_h|\lesssim h^2+|\mathbf{A}_h-\bP_h| \lesssim h^2.
\end{equation*}
This yields
\begin{equation}\label{aux16}
 I_3\lesssim  h^2 \|\nath v_h\|_{L^2(\Gamma_h)}\| u^e\|_{L^2(\Gamma_h)} \lesssim  {h^2} \|u\|_{H^1(\Gamma)}\|v_h\|_{H^1(\Gamma_h)}.
\end{equation}
Combining the results \eqref{aux12}-\eqref{aux16} proves the lemma.
\end{proof}

\subsection{Interpolation bounds}
For a smooth function $v$ defined on $\Gamma$ consider its extension on $\Gamma_h$, $v^e$. In this subsection, we show
that $v^e$ can be approximated using the finite element trace space $V_h^\Gamma$ with optimal order accuracy.
First we need the following simple lemma.

\begin{lemma} Consider an arbitrary plane  $\mathbb{P}\subset\mathbb{R}^3$ and any cubic cell $S\in\mathcal{T}_h$.
It holds
\begin{equation}\label{eq_trace}
\|v\|_{L^2(S\cap \mathbb{P})}^2\lesssim h^{-1}_S\|v\|_{L^2(S)}^2+h_S\|\nabla v\|_{L^2(S)}^2\quad \forall~v\in H^1(S).
\end{equation}
\end{lemma}
\begin{proof} The proof follows from the fact that any cubic cell is divided into a finite number of regular tetrahedra
and further applying Lemma~4.2 from \cite{Hansbo2003} on each of these tetrahedra.
\end{proof}

Now we are ready to prove the trace finite element interpolation bounds.
\begin{theorem}\label{thm_int} For arbitrary $v\in H^2(\Gamma)$ the following estimate holds with a constant $c$ independent of $h$, $v$, and on
how $\Gamma_h$ intersects $\mathcal{T}_h$:
 \begin{equation}\label{eq_interp}
\inf_{v_h\in V^\Gamma_h}\left(\|v^e-v_h\|_{L^2(\Gamma_h)} + h\|\nabla_{\Gamma_h}(v^e-v_h)\|_{L^2(\Gamma_h)}\right) \le c h^{2}\|v\|_{H^2(\Gamma)}.
\end{equation}
\end{theorem}
\begin{proof}
Assume   the maximum cells size $h$ in $\mathcal{T}_h$ is attained at level $k$ of the octree. Then the cells at the level $k$ provide uniform cartesian subdivision which covers $\Omega$. Denote this subdivision by $\mathcal{T}_h^k$ and the
set of corresponding nodes by $\mathcal{N}_h^k$.
Note that $v\in H^2(\Gamma)$ and $\Gamma\in C^2$ imply $v^e\in H^2(U_h)$ for the normal extension and we may define the nodal piecewise  trilinear interpolant $I_h(v^e)$ such that $I_h(v^e)(\bx)=v^e(\bx)$ for $\bx\in \mathcal{N}_h^k\cap U_h$ and
$I_h(v^e)(\bx)=0$ in other nodes. Since $\mathcal{T}_h^k$ does not contain hanging  nodes, $I_h(v^e)$ is continuous and so it holds $I_h(v^e)\in V_h$.

Using the result in \eqref{eq_trace} and the interpolation properties of $I_h(v^e)$, we get for $v_h=I_h(v^e)|_{\Gamma_h}$:
\[
\begin{split}
\|v^e-v_h\|_{L^2(\Gamma_h)}^2 + h^2\|\nabla_{\Gamma_h}(v^e-v_h)\|_{L^2(\Gamma_h)}^2
= \sum_{T\in\mathcal{F}_h} \|v^e-v_h\|_{L^2(T)}^2 + h^2\|\nabla_{\Gamma_h}(v^e-v_h)\|_{L^2(T)}^2\\
\le \sum_{T\in\mathcal{F}_h} h_{S_T}^{-1}\|v^e-I_h(v^e)\|_{L^2(S_T)}^2 + h^2h_{S_T}^{-1}\|\nabla(v^e-I_h(v^e))\|_{L^2(S_T)}^2
+h^3\|\nabla^2(v^e-I_h(v^e))\|_{L^2(S_T)}^2\\
\le C\,h^4 h_{S_T}^{-1}\sum_{T\in\mathcal{F}_h} \|v^e\|_{H^2(S_T)}^2.
\end{split}
\]
The assumption that the number of local refinement levels is bounded yields $h h_{S_T}^{-1}\le C$.
Hence, it holds
\[
\|v^e-v_h\|_{L^2(\Gamma_h)}^2 + h^2\|\nabla_{\Gamma_h}(v^e-v_h)\|_{L^2(\Gamma_h)}^2\le h^3 \|v^e\|_{H^2(U_h)}^2.
\]
It remains to apply \eqref{e:3.14}.
\end{proof}

\subsection{Error estimates} \label{sectmain}

Now we combine the results derived in the previous sections to prove the first error estimate.
\begin{theorem}\label{Th1}
Assume  $u$ is the solution of \eqref{problem} and let $u_h$ be the discrete solution of the trace finite element method \eqref{FEM}. Then the following estimate holds:
\begin{equation}\label{errorEst}
\|u^e-u_h\|_{H^1(\Gamma_h)}\lesssim h \|f\|_{L^2(\Gamma)}.
\end{equation}
\end{theorem}
\begin{proof} Consider an interpolant $(I_h u^e)\in V_h$ from Theorem~\ref{thm_int}. The triangle inequality yields
\begin{equation}\label{triang}
\|u^e-u_h\|_{H^1(\Gamma_h)}\leq \|u^e-(I_h u^e)|_{\Gamma_h}\|_{H^1(\Gamma_h)} + \|(I_h u^e)|_{\Gamma_h}- u_h \|_{H^1(\Gamma_h)}.
\end{equation}
The second term in the upper bound can be estimated using coercivity, continuity of the finite element bilinear form,
interpolation properties from Theorem~\ref{thm_int}, and the consistency estimate in Lemma~\ref{thm3}:
\begin{align*}
  \|&\ipue- u_h \|^2_{H^1(\Gamma_h)} \lesssim a_h(\ipue- u_h,\ipue- u_h )\nonumber \\
&= a_h(\ipue -u^e, \ipue- u_h)+a_h(u^e-u_h, \ipue-u_h)\nonumber \\
&\lesssim h \|u\|_{H^2(\Gamma)}\|\ipue-u_h\|_{H^1(\Gamma_h)}+ |a_{h}(u^e,\ipue-u_h)-f_{h}(\ipue-u_h)|\nonumber \\
&\lesssim h \big(\|u\|_{H^2(\Gamma)}+\|f\|_{L^2(\Gamma)}\big)\|\ipue-u_h\|_{H^1(\Gamma_h)}.\nonumber
 \end{align*}
This results in
\begin{equation}\label{aux24}
  \|\ipue- u_h \|_{H^1(\Gamma_h)}\lesssim  h (\|u\|_{H^2(\Gamma)}+\|f\|_{L^2(\Gamma)}).
\end{equation}
The error estimate  \eqref{errorEst} follows from \eqref{triang}, \eqref{aux24}, \eqref{eq_interp}, and $\|u\|_{H^2(\Gamma)}\lesssim\|f\|_{L^2(\Gamma)}$.\quad
\end{proof}

We now apply a duality argument to obtain an $L^2(\Gamma_h)$ error bound.

\begin{theorem} \label{thmL2}
Additionally assume that \eqref{A1} and \eqref{A1h} hold with $\bw$ replaced by $-\bw$.
With the same  $u$, $u_h$  as in Theorem~{\rm \ref{Th1}}, the following error bound holds
\begin{equation}\label{L2_est}
\|u^e-u_h\|_{L^2(\Gamma_h)}\lesssim h^2\,\|f\|_{L^2(\Gamma)}.
\end{equation}
\end{theorem}

\begin{proof}
Denote $e_h:=(u^e-u_h)|_{\Gamma_h}$ and let  $e^l_h$ be  the
lift of $e_h$ on $\Gamma$.  Consider the problem: Find $w\in
H^1(\Gamma)$, such that
\begin{equation}\label{dual}
a(v,w)= \int_\Gamma e_h^l v\, \rd\bs\qquad\text{for all}~~
v\in H^1(\Gamma).
\end{equation}
The problem is well-posed and the solution $w$ satisfies  $w\in H^2(\Gamma)$ and $\|w\|_{H^2(\Gamma)}\lesssim\|e^l_h\|_{L^2(\Gamma)}$.
Furthermore, $\|w^e\|_{H^1(\Gamma_h)} \lesssim\|w\|_{H^1(\Gamma)}\lesssim \|e^l_h\|_{L^2(\Gamma)}$. Due to \eqref{dual}  we have, for any  $\psi_h\in V_h^\Gamma$,
\begin{equation}\label{est_i0}
\begin{split}
\|e^l_h\|^2_{L^2(\Gamma)}&=a(e^l_h,w)=a(e^l_h,w)-a_h(e_h,w^e)+a_h(e_h,w^e-\psi_h)+a_h(e_h,\psi_h)\\
&=\left[a(e^l_h,w)-a_h(e_h,w^e)\right]+a_h(e_h,w^e-\psi_h)+\left[a_h(u^e,\psi_h)-f_h(\psi_h)\right].
\end{split}
\end{equation}
We let $\psi_h=(I_h u^e)\in V_h$, interpolant given by Theorem~\ref{thm_int}.
The terms on the right hand side of \eqref{est_i0} are treated separately. The bound for the third one follows from
Lemma~\ref{thm3}:
\begin{equation}\label{est_i1}
a_h(u^e,\psi_h)-f_h(\psi_h)\lesssim h^2(\|u\|_{H^1(\Gamma)}+\|f\|_{L^2(\Gamma)})\|\psi_h\|_{H^1(\Gamma_h)}\lesssim h^2(\|u\|_{H^1(\Gamma)}+\|f\|_{L^2(\Gamma)})\|u^e\|_{H^2(\Gamma)}.
\end{equation}
The bound for the second term follows from the continuity of the bilinear form, interpolation properties in Theorem~\ref{thm_int} and the error bound in Theorem~\ref{Th1}:
\begin{equation}\label{est_i2}
a_h(e_h,w^e-\psi_h)\lesssim \|e_h\|_{H^1(\Gamma_h)}\|w^e-\psi_h\|_{H^1(\Gamma_h)}
\lesssim h^2\|f\|_{L^2(\Gamma)}\|w^e\|_{H^2(\Gamma)}.
\end{equation}
For the first term we get the expression similar to \eqref{aux12}:
\begin{multline*}
a(e^l_h,w)-a_h(e_h,w^e) =\eps\int_{\Gamma_h}(\mathbf{A}_h-\bP_h)\nabla_{\Gamma_h} w^e \cdot \nabla_{\Gamma_h}e_h\, \ds_h\\
-\int_{\Gamma_h}((\mathbf{B}_h-\bP_h)\mathbf{w}^e\cdot\nath w^e) e_h\, \ds_h + \int_{\Gamma_h}(\mu_h-1)c^ew^e e_h\, \ds_h.
\end{multline*}
Hence, repeating the same argument as in the proof of Lemma~\ref{thm3}, we get
\begin{equation*}
a(e^l_h,w)-a_h(e_h,w^e)\lesssim h^2\|e_h\|_{H^1(\Gamma_h)}\|w^e\|_{H^1(\Gamma)}.
\end{equation*}
Further, employing the error bound from Theorem~\ref{Th1}, we have
\begin{equation}\label{est_i3}
a(e^l_h,w)-a_h(e_h,w^e)\lesssim h^3\|f\|_{L^2(\Gamma)}\|w^e\|_{H^1(\Gamma)}.
\end{equation}
Combining \eqref{est_i0}--\eqref{est_i3} with the regularity result for the solution of the
dual problem \eqref{dual} completes the proof.\qquad
\end{proof}

\subsection{A posteriori estimate and error indicator}\label{s_adapt}

In this section, we deduce  an a posteriori error estimate for the trace FEM. Based on this estimate we define a residual
type error indicator, which we use for a mesh adaptation purpose.  We will treat only the formulation with surface gradient \eqref{FEM} and assume that the Peclet number is sufficiently small and hence no additional stabilization is needed.

Consider the surface finite element error $e_h=u^e-u_h$ on $\Gamma_h$. We prove an a posteriori bound for the lift of $e_h$ on $\Gamma$, i.e.  $e_h^l=u-u_h^l$ on  $\Gamma$. Thanks to \eqref{A1} the functional $\|[v]\|:=(\eps \|\nabla_\Gamma v\|_{L_2(\Gamma)}^2+\|(c+\frac12\Div_\Gamma\bw)\, v\|_{L_2(\Gamma)}^2)^{\frac12}$ defines a norm of $V$ such that
\begin{equation}
\label{3.1}
\|[ e_h^l]\|\le\sup_{\psi:~\|[ \psi]\|=1} \int_\Gamma a( e_h^l,  \psi) \d \bs.
\end{equation}
Now we let $\psi \in V$ with $\|[ \psi]\|=1$.
For arbitrary $\psi_h\in V_h^\Gamma$, observe the identities
\begin{equation}\label{s4_e1}
\begin{split}
a( e_h^l,  \psi)&= \int_\Gamma f\psi\,d\bs-a(u_h^l,  \psi)= \int_{\Gamma_h} f^e\mu_h(\psi^e-\psi_h)\,d\bs_h+\int_{\Gamma_h} (f^e\mu_h-f^e)\psi_h\,d\bs_h+a_h(u_h,  \psi_h)-a(u_h^l,  \psi)\\
&=\int_{\Gamma_h} f^e\mu_h(\psi^e-\psi_h)\,d\bs_h +\int_{\Gamma_h} f^e(\mu_h-1)\psi_h\,d\bs_h+a_h(u_h,  \psi_h-\psi^e)
\\&\qquad -\left[\eps\int_{\Gamma_h}(\mathbf{A}_h-\bP_h)\nabla_{\Gamma_h} u_h \cdot \nabla_{\Gamma_h}\psi^e\, \ds_h-\frac12\int_{\Gamma_h}((\mathbf{B}_h-\bP_h)\mathbf{w}^e\cdot\nath \psi^e)  u_h\, \ds_h\right. \\
  &\qquad\qquad\left.+ \int_{\Gamma_h}(\mu_h-1)c^eu_h \psi^e\, \ds_h\right].
\end{split}
\end{equation}

We apply elementwise integration by parts to the third term on the righthand side of \eqref{s4_e1}:
\begin{equation}\label{s4_e2}
\begin{split}
a_h(u_h,  \psi_h-\psi^e)&= \int_{\Gamma_h} (\eps\Delta_{\Gamma_h}  u_h-(c_h+\Div_{\Gamma_h}\bw_h) u_h - \mathbf{w}_h\cdot\nath u_h)) (\psi^e-\psi_h) \d \bs_h\\
&\qquad -\frac12\sum_{T\in \F_h}\int_{\partial T}
\eps\llbracket \unablah u_h \rrbracket (\psi^e-\psi_h) \d\br + \frac12\sum_{T\in\F_h}\int_{\partial T}( \bw_h\cdot\llbracket\mathbf{m}_h\rrbracket)u_h  (\psi^e-\psi_h)\d\br.
\end{split}
\end{equation}
Substituting \eqref{s4_e2} into \eqref{s4_e1} and applying the Cauchy inequality elementwise over $\F_h$, we get
\begin{equation}\label{s4_e3}
\begin{split}
a( e_h^l,  \psi)&\lesssim
\left(\sum_{T\in \F_h} |1-\mu_h|^2_{L^\infty(T)}\left(\|f^e\|^2_{L^2(T)}+\|u_h\|^2_{L^2(T)}\right)+
\|\mathbf{A}_h-\bP_h\|^2_{L^\infty(T)}\|\nabla_{\Gamma_h} u_h\|^2_{L^2(T)}\right.\\
&\left.\phantom{\sum_{T\in \F_h}}+
\|\mathbf{B}_h-\bP_h\|^2_{L^\infty(T)}\|u_h\|^2_{L^2(T)}
\right)^{\frac12}\|\psi^e\|_{H^1(\Gamma_h)}
\\&+
\left(\sum_{T\in \F_h}\left[\eta_R(T)^2+\eta_E(T)^2 \right]\right)^{\frac12}
\left(\sum_{T\in \F_h}\left[h^{-2}_{S_T}\|\psi^e-\psi_h\|_{L^2(T)}^2
+\sum_{e\in\partial T}h_{S_T}^{-1}\|\psi^e-\psi_h\|_{L^2(e)}^2\right]
\right)^{\frac12},
\end{split}
\end{equation}
with
\[
\begin{split}
\eta_R(T)^2&=h^2_{S_T}\|f_h+\eps\Delta_{\Gamma_h}  u_h-(c_h+\Div_{\Gamma_h}\bw_h) u_h - \mathbf{w}_h\cdot\nath u_h\|^2_{L^2(T)}\\
\eta_E(T)^2&=\sum_{e\in\partial T}h_{S_T}\left(\|\llbracket \eps\unablah u_h \rrbracket\|_{L^2(e)}^2+\|\bw_h\cdot\llbracket\mathbf{m}_h\rrbracket\|_{L^2(e)}^2\right)
\end{split}
\]
Combining \eqref{3.1} and \eqref{s4_e3} leads to an a posteriori error estimate, where $\psi$-dependent terms on the
righthand side should be handled further.

Regarding $\psi$-dependent terms in \eqref{s4_e3} we note the following: due to \eqref{e:3.16} and \eqref{grad2}
we have $\|\psi^e\|_{H^1(\Gamma_h)} \lesssim \|\psi\|_{H^1(\Gamma)} $. For the second term, in \cite{DemlowOlsh} it was proved
that there exists such $\psi_h\in V_h$ such that
\begin{equation}\label{A2}
\sum_{T\in \F_h}\left[h^{-2}_{S_T}\|\psi^e-\psi_h\|_{L^2(T)}^2
+\sum_{e\in\partial T}h_{S_T}^{-1}\|\psi^e-\psi_h\|_{L^2(e)}^2\right]\lesssim \|\psi\|_{H^1(\Gamma)}.
\end{equation}
In \cite{DemlowOlsh}, the estimate \eqref{A2} was proved for a regular (locally refined) tetrahedra bulk triangulation and $P_1$ elements. To extend this result to octree cubic grids, one has to deal with certain technical difficulties related
to the construction of a Scott-Zhang type interpolant, handling algebraic constraints in hanging nodes and hierarchical basis functions.
Such analysis will be done elsewhere. Let us \textit{assume} \eqref{A2} to be true for octree  grids and
continuous piecewise trilinear functions.

Finally, if we assume that local grid refinement leads to better local surface reconstruction, i.e. \eqref{e:3.9}
and \eqref{e:3.10} can be formulated locally, then the first term on the right-hand side of \eqref{s4_e3} is heuristically of higher order. We may assume $|1-\mu_h|^2_{L^\infty(T)}+\|\mathbf{A}_h-\bP_h\|^2_{L^\infty(T)}+
\|\mathbf{B}_h-\bP_h\|^2_{L^\infty(T)}\lesssim h^4_{S_T}\|\bH\|^2_{L^\infty(\bp(T))}.$ This suggests introducing also
the geometric error indicator
\[
\eta_G(T)^2:=h^4_{S_T}\|\bH_h\|^2_{L^\infty(T)}\left(\|f^e\|^2_{L^2(T)}+\|u_h\|^2_{H^1(T)}\right),
\]
where $\bH_h$ is any Hessian reconstruction on $\Gamma_h$.

Thus, for the purpose of local mesh adaptation we use the following error indicator:
\begin{equation}\label{indicator}
\eta(T):=(\alpha_r\eta_R(T)^2+\alpha_e\eta_E(T)^2+\alpha_g\eta_G(T)^2)^{\frac12}.
\end{equation}
with some parameters $\alpha_r,\alpha_e,\alpha_g\ge 0$.
We set $\alpha_g$ equal to 0 or 1 depending on whether we wish to account (explicitly) for geometric errors or not. For regular elliptic
problems we set $\alpha_r=\alpha_e=1$, while for transport dominated problems more weight is put
on the edge error indicator (see next section).
We also note that the second edge term in the definition of $\eta_E(T)^2$ is of higher order. Our numerical experience suggests that it can excluded
without any notable change in the performance of an adaptive method.

Although we deduced the indicator based on partially heuristic arguments, numerical experiments with adaptive method
based on $\eta(T)$ and standard ``maximum'' marking strategy  show optimal convergence of the error in energy and $L^2$ norms.

\section{Numerical examples}\label{s_numer}

\subsection{Smooth solutions on a sphere and a torus}

\begin{example}\label{exam1}\rm
 As a first test problem we consider the
Laplace--Beltrami type equation on the unit sphere:
\[
-\Delta_\Gamma u +u =f\quad\mbox{on}~\Gamma,
\]
with $\Gamma=\{ \bx\in\R^3\mid \|\bx\|_2 = 1\}$ and $\Omega= (-2,\,2)^3$.
The source term $f$ is taken such that the solution is given by
\[
    u(\bx)= \frac{a}{\|\bx\|^3}\left(3x_1^2x_2 - x_2^3\right),\quad
    \bx=(x_1,x_2,x_3)\in\Omega,
\]
with $a=12$.
Using the representation of $u$ in spherical coordinates one can
verify  that $u$ is an eigenfunction of $-\Delta_\Gamma$:
$
    u( r, \phi, \theta) = a\sin(3\phi)\sin^3\theta$,
$ -\Delta_\Gamma u+u = 13 u =: f(r,\phi,\theta).
$
Both $u$ and $f$ are constant along normals at $\Gamma$.
\end{example}

\begin{figure}[ht!]
\begin{center}
  \includegraphics[width=0.45\textwidth]{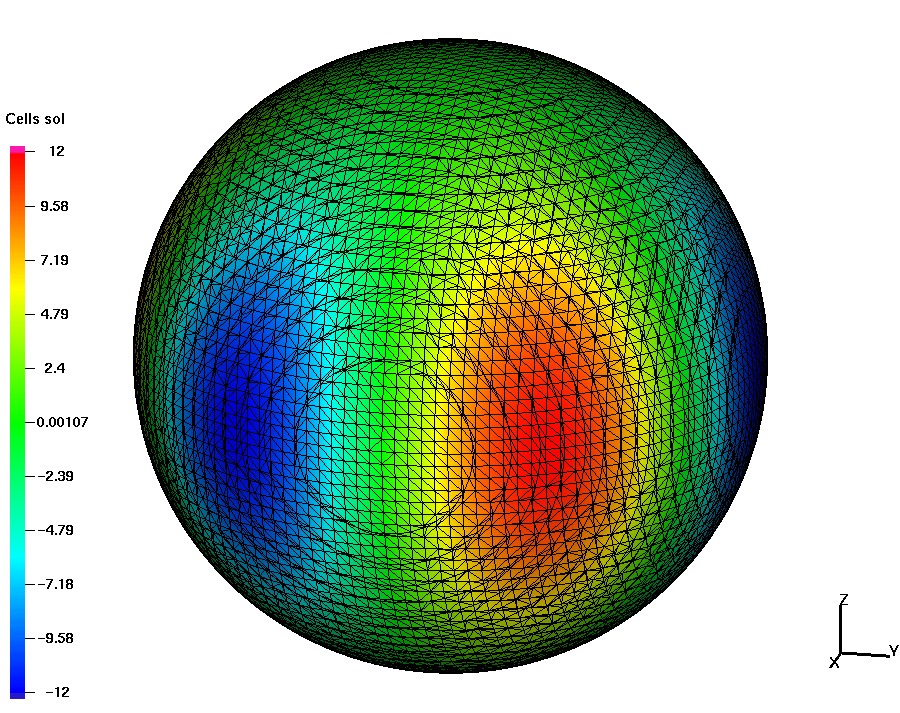}\qquad
  \includegraphics[width=0.45\textwidth]{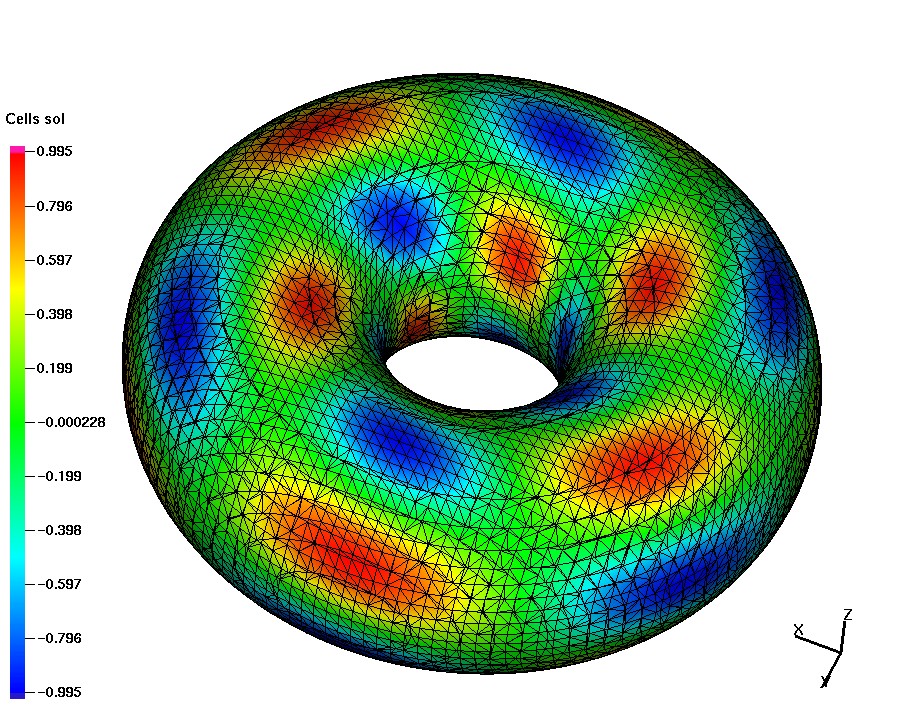}
\caption{\label{fig:curv1} The  numerical solutions and surface meshes from Examples~\ref{exam1} and \ref{exam2}.}
\end{center}
\end{figure}

\begin{example}\label{exam2}\rm
 In the second example, we consider a  torus instead of the unit sphere and the same equation:
$\Gamma\subset\Omega=(-2,2)^3$, with
$\Gamma = \{ \bx\in\Omega \mid r^2 = x_3^2 + (\sqrt{x_1^2 + x_2^2} - R)^2\}$.
We take $R= 1$ and $r= 0.6$. In the coordinate system $(\rho, \phi,
\theta)$, with
\[
    \bx = R\begin{pmatrix}\cos\phi \\ \sin\phi \\ 0 \end{pmatrix}
    + \rho\begin{pmatrix}\cos\phi\cos\theta \\ \sin\phi\cos\theta \\
    \sin\theta \end{pmatrix},
\]
the $\rho$-direction is normal to $\Gamma$, $\frac{\partial \bx}{\partial\rho}\perp\Gamma$ for $\bx\in\Gamma$.
Thus, the following solution $u$ and corresponding right-hand side
$f$ are constant in the normal direction:
\begin{equation}\label{polar2}
  \begin{split}
    u(\bx)&= \sin(3\phi)\cos( 3\theta + \phi),\\
    f(\bx)&= r^{-2} (9\sin( 3\phi)\cos( 3\theta + \phi))  - (R + r\cos( \theta))^{-2}(-10\sin( 3\phi)\cos(3\theta + \phi) - 6\cos( 3\phi)\sin( 3\theta + \phi)) \\
          & \quad -(r(R + r\cos( \theta))^{-1}(3\sin( \theta)\sin(
          3\phi)\sin( 3\theta + \phi))+u(\bx).
  \end{split}
\end{equation}
\end{example}

We compute numerical solutions to Examples~\ref{exam1} and~\ref{exam2} using the trace FE methods \eqref{FEM} and~\eqref{FEM_fg}
on a sequence of octree bulk grids. The initial grid was uniform with $h=\frac14$. Further the grid was gradely refined towards the surfaces. All linear algebra systems in this and further experiments were solved with the help of PETSc library: We computed LU factorizations of diagonally scaled stiffness matrices.
Finite element errors and convergence rates are shown in Table~\ref{tab1}. Both variants
demonstrate second order convergence in $L^2$ surface norm and close to second order in $L^\infty$ surface norm.
Computed solutions and final meshes are visualized on Figure~\ref{fig:conv1}.

\begin{table}
\begin{center}
\caption{Surface gradient and full gradient trace FEM convergence for uniform mesh
refinement and smooth solutions. \label{tab1}}
\small
\begin{tabular}{rr|llll|llll}\hline
&&\multicolumn{4}{|c|}{surface gradient variant \eqref{FEM}} &   \multicolumn{4}{|c}{full gradient variant~\eqref{FEM_fg}} \\
&\#d.o.f. & $L^2$-norm & rate & $L^\infty$-norm& rate &  $L^2$-norm & rate & $L^\infty$-norm& rate \\ \hline\\[-2ex]
sphere
&292       & 2.672e-1 &      & 6.631e-1 &       & 6.041e-1 &      & 1.257e-0 & \\
&1398	   & 5.813e-2 & 2.20 & 1.462e-1 & 2.18  & 1.418e-1 & 2.09 & 3.959e-1 & 1.67 \\
&5960      & 1.366e-2 & 2.09 & 4.032e-2 & 1.86  & 3.566e-2 & 1.99 & 1.031e-1 & 1.94 \\
&24730     & 3.364e-3 & 2.02 & 1.138e-2 & 1.86  & 8.891e-3 & 2.00 & 2.582e-2 & 2.00 \\    \hline\\[-2ex]
torus
& 297      & 2.415e-1 & & 7.898e-1 & & 3.199e-1 & & 8.861e-1 & \\
&1289      & 5.047e-2 & 2.26 & 2.123e-1 & 1.90 & 8.406e-2 & 1.93 & 4.437e-1 & 0.98 \\
&5001      & 1.049e-2 & 2.27 & 5.400e-2 & 1.98 & 1.956e-2 & 2.10 & 8.562e-2 & 2.37 \\
&20073     & 2.367e-3 & 2.15 & 1.341e-2 & 2.01 & 4.979e-3 & 1.97 & 2.184e-2 & 1.97 \\ \hline
\end{tabular}
\end{center}
\end{table}

\subsection{Smooth solutions on more complicated geometries}

To demonstrate the flexibility of the method with respect to  the form of $\Gamma$, we consider the Laplace-Beltrami equation on two geometrically more complicated surfaces. Both examples of surfaces can be found in
\cite{DE2013}.

\begin{figure}[ht!]
\begin{center}
  \includegraphics[width=0.45\textwidth]{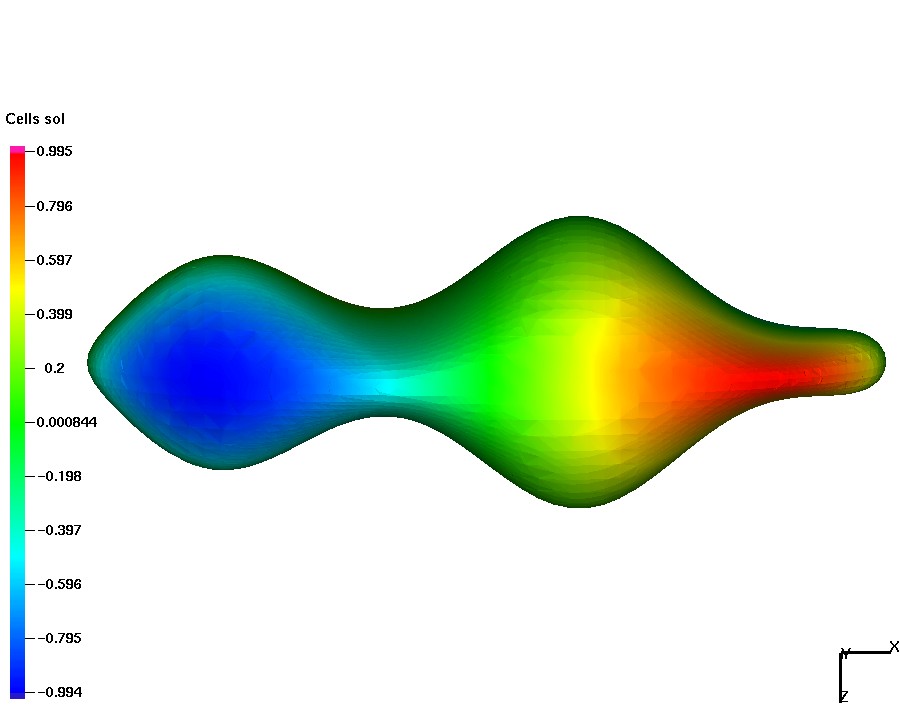}\qquad
  \includegraphics[width=0.45\textwidth]{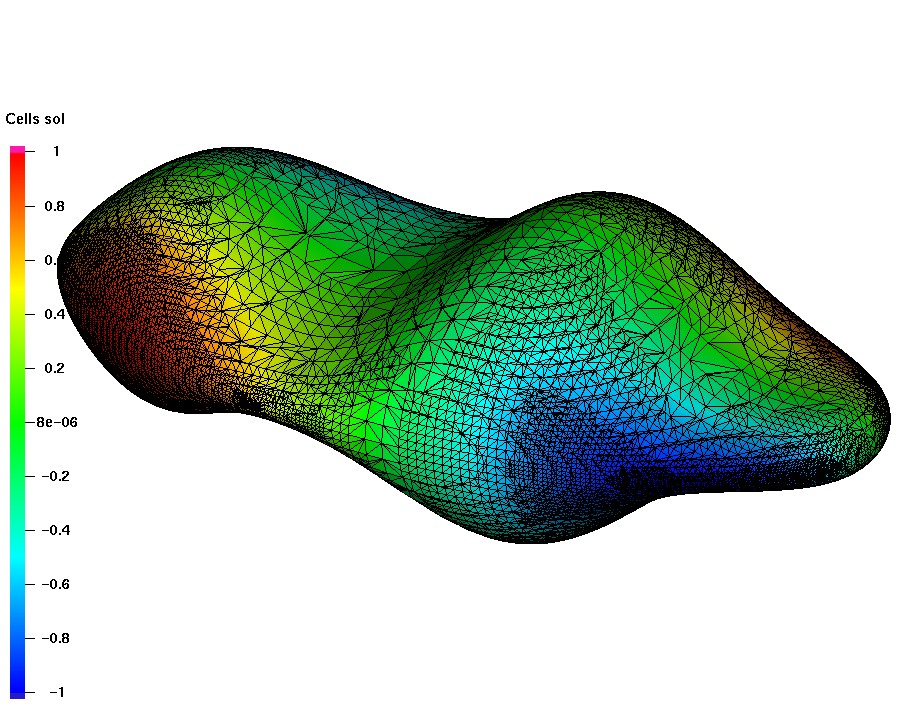}
\caption{\label{fig:curv2} Illustration of the surface and solution from Example~\ref{exam3}. The right figure shows also triangulation after 7 steps of adaptation based in the error indicator.}
\end{center}
\end{figure}

\begin{figure}[ht!]
\begin{center}
  \includegraphics[width=0.4\textwidth]{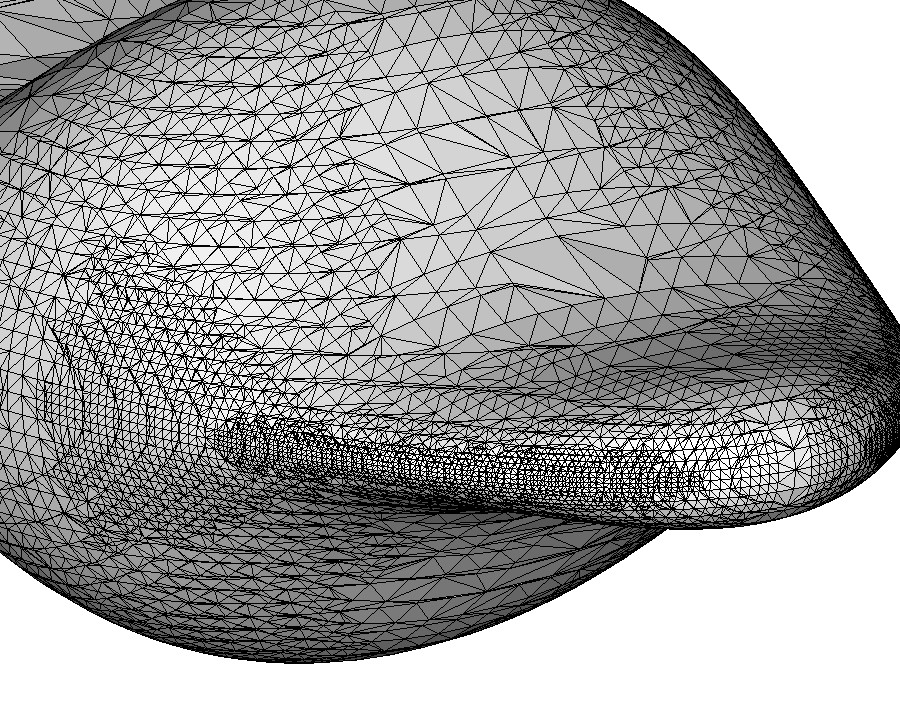}\qquad\qquad
  \includegraphics[width=0.45\textwidth]{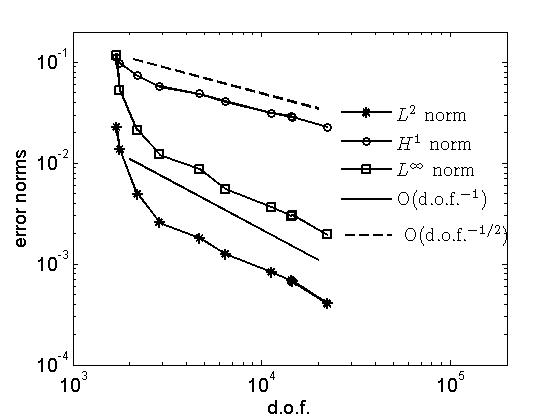}
\caption{\label{fig:curv22} Left: Adaptive trace FEM convergence in Example~\ref{exam3}. Right: The zoom of the trace surface mesh in Example~\ref{exam3}. The adaptation was based on the residual  indicator accounting for geometric errors. }
\end{center}
\end{figure}

\begin{example}\label{exam3}\rm
The fist surface is given by the zero of the level set function
\[
\phi(\bx)=\frac14 x_1^2+x_2^2+\frac{4x_3^2}{(1+\frac12\sin(\pi x_1))^2}-1.
\]
The exact solution to \eqref{problem} with $\bw=0$, $\ep=1$, $c=1$, is taken very regular, $u=x_1x_2$ on $\Gamma$. The surface is illustrated in Figure~\ref{fig:curv2}. The normal vector on $\Gamma$ is $\bn=\nabla\phi/|\nabla\phi|=:(n_1,n_2,n_3)^T$ and  $\bH$ can be computed from $\bH=\nabla_\Gamma\bn$. The entries of $\bH$ are given by
\[
 \bH_{im}=\frac1{|\nabla\phi|}\left(\phi_{x_ix_m}
-\frac{\phi_{x_i}(\sum_{k=1}^3\phi_{x_k}\phi_{x_kx_m})+\phi_{x_m}(\sum_{j=1}^3\phi_{x_j}\phi_{x_ix_j})}{|\nabla\phi|^2}
+\frac{\phi_{x_i}\phi_{x_m}(\sum_{j=1}^3\sum_{k=1}^3\phi_{x_k}\phi_{x_kx_j})}{|\nabla\phi|^4}\right)
\]
for $i,m=1,2,3$. The right-hand side is then given by
\[
f(\bx)=x_1x_2+2n_1(\bx)n_2(\bx)+\mbox{tr}(\bH(\bx))(x_1n_2(\bx)+x_2n_1(\bx)),\quad\bx\in\Gamma.
\]
Figure~\ref{fig:curv2} shows the numerical solution and plots it over the trace mesh, which results after 9 steps of adaptation, starting with the uniform cubic mesh with $h=1/4$.
Here and in any further adaptive mesh refinement we employ a ``maximum'' marking strategy in which  all volume
cubes $S$ from $\omega_h$ with $\eta(S)>\frac12\max_{S\in\omega_h}\eta(S)$ are marked for further refinement.
Note that some adjunct cubes also may need refinement if one wishes to keep the octree balanced.
Figure~\ref{fig:curv22} (right) displays the  finite element error reduction if the mesh adaptation process is based
on the error indicator \eqref{indicator}. The results demonstrate optimal convergence order of the adaptive trace FEM
with respect to the total number of degrees of freedom in $L^2$, $H^1$ and $L^\infty$ surface norms. The results are shown for the surface gradient variant of the method \eqref{FEM}. Here and further ``number of d.o.f.'' means only the number of active degrees of freedom, which is equal to the dimension of the resulting system of linear algebraic equations. We set $\alpha_g=1$  in the error indicator  to account for geometric errors. The left plot in Figure~\ref{fig:curv22} zooms the trace surface mesh.  From this plot we see that the mesh refinement generally happens in regions with higher surface curvatures.
\end{example}

\begin{figure}[ht!]
\begin{center}
  \includegraphics[width=0.5\textwidth]{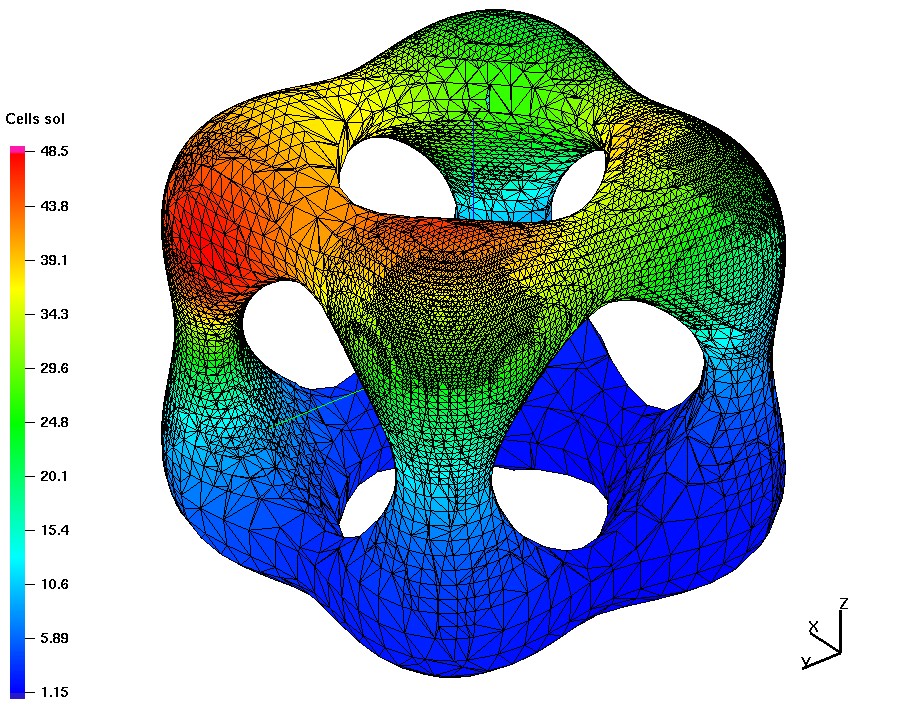}\qquad
\caption{\label{fig:curv4} The surface, numerical solution and adaptive mesh from Example~\ref{exam4}.}
\end{center}
\end{figure}
\begin{example}\label{exam4}\rm This is another example of a more complicated domain, which  is homeomorphic to the sphere
with 6 handles, see Figure~\ref{fig:curv4}. The surface is given implicitly as the zero level set of
\[
\phi=(x_1^2+x_2^2-4)^2+(x_2^2-1)^2+(x_2^2+x_3^2-4)^2+(x_1^2-1)^2+(x_1^2+x_3^2-4)^2+(x_3^2-1)^2-13.
\]
We solve the Laplace-Beltrami equation with right-hand side $f=100\sum_{j=1}^4\exp{(-|\bx-\bx^j|^2)}$, with
\[
\bx^1=(-1,1,2.04),~~  \bx^2=(1,2.04,1),~~ 
\bx^3=(2.04,0,1),~~  \bx^4=(-0.-1,-2.04).
\]
The points are close to the surface and the right-hand side is varying rapidly in vicinities of these 4 points, and hence the same is expected from the solution. The solution and the grid resulted after 12 steps of refinement are visualized in Figure~\ref{fig:curv4}. The refinement was based on the error indicator \eqref{indicator} with $\alpha_g=1$.
\end{example}

\subsection{Laplace-Beltrami problem with point singularity}
\begin{figure}[ht!]
\begin{center}
  \includegraphics[width=0.45\textwidth]{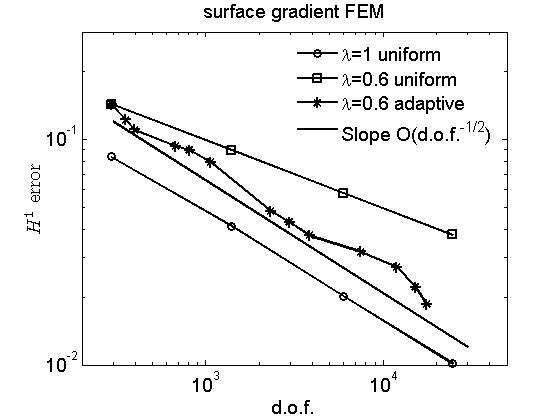}\qquad
  \includegraphics[width=0.45\textwidth]{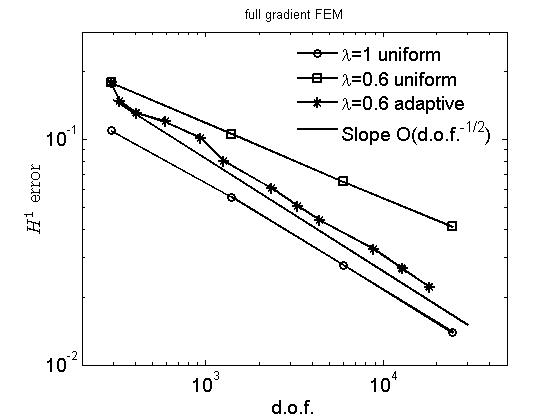}\\
   \includegraphics[width=0.45\textwidth]{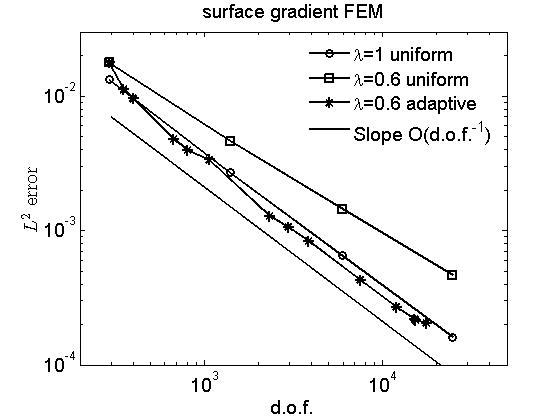}\qquad
  \includegraphics[width=0.45\textwidth]{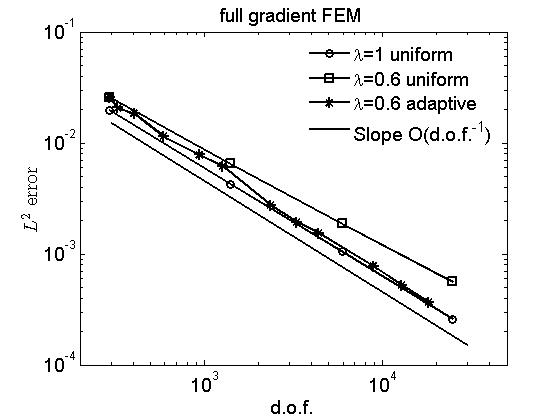}
\caption{\label{fig:lambda} Decrease of the error in $H^1(\Gamma_h)$-norm (upper plots), $L^2(\Gamma_h)$-norm (bottom plots)   for the adaptive algorithm and Example~\ref{exam5}.
Left column plots show results for the surface gradient  formulation \eqref{FEM}, while the right  column plots show results for the full gradient  formulation \eqref{FEM_fg}.}
\end{center}
\end{figure}

\begin{example}\label{exam5}\rm
For the next test problem we consider the Laplace-Beltrami equation on the unit sphere.
The solution and the source term  in spherical coordinates are given by
\begin{equation}\label{polar}
    u= \sin^\lambda\theta\sin\phi,\qquad  f=(1+\lambda^2+\lambda)\sin^\lambda\theta\sin\phi+(1-\lambda^2)\sin^{\lambda-2}\theta\sin\phi.
\end{equation}
One verifies
\[
\nabla_\Gamma u=\sin^{\lambda-1}\theta\left(\frac12\sin2\phi(\lambda\cos^2\theta-1),\, \sin^2\phi(\lambda\cos^2\theta-1)+1,\,-\frac12\lambda\sin2\theta\sin\phi\right)^T.
\]
 For $\lambda<1$ the solution $u$ is singular at the north and south poles of the sphere so that $u\in H^1(\Gamma)$, but $u\notin H^2(\Gamma)$. Following \cite{DD07}  we set $\lambda=0.6$ to model the point singularity  and
 contrast it to the regular problem with $\lambda=1$.

 Results produced by the adaptive algorithm   are shown in Figure~\ref{fig:lambda}, where they are compared to
 the results for uniform grid refinement. As expected, the regular refinement leads to a suboptimal convergence
 for the singular case of $\lambda=0.6$.  Adaptive refinement driven by the error indicator~\eqref{indicator} leads to
 optimal convergence rates in $L^2$ and $H^1$ surface norms.
 Note that reliability of the error indicator for $H^1$ error norm was proved in \cite{DemlowOlsh} for tetrahedral meshes. A posteriori error analysis in other norms remains an open question.

  The left column in Figure~\ref{fig:lambda}  displays error decrease for the surface gradient  formulation \eqref{FEM}, while the right  column of plots displays error decrease for  the full gradient  formulation \eqref{FEM_fg}. Similar to regular problems and regular refinement in Examples~\ref{exam1} and~\ref{exam2}, the adaptive algorithm shows close performance for both formulations  producing slightly more accurate results for the surface gradient formulation \eqref{FEM}.

\begin{figure}[ht!]
\begin{center}
  \includegraphics[width=0.35\textwidth]{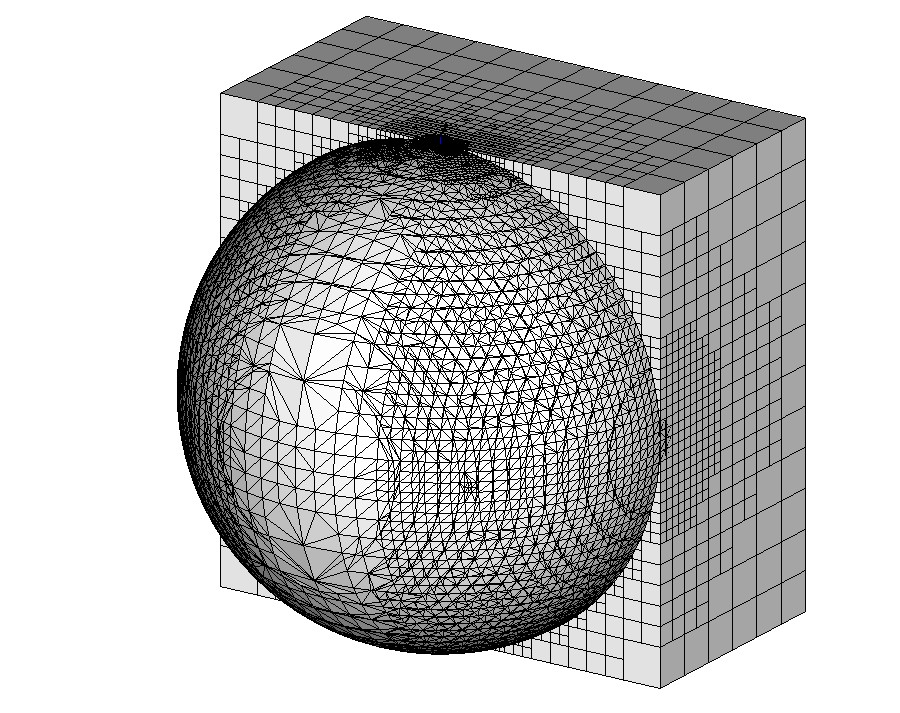}\quad
  \includegraphics[width=0.35\textwidth]{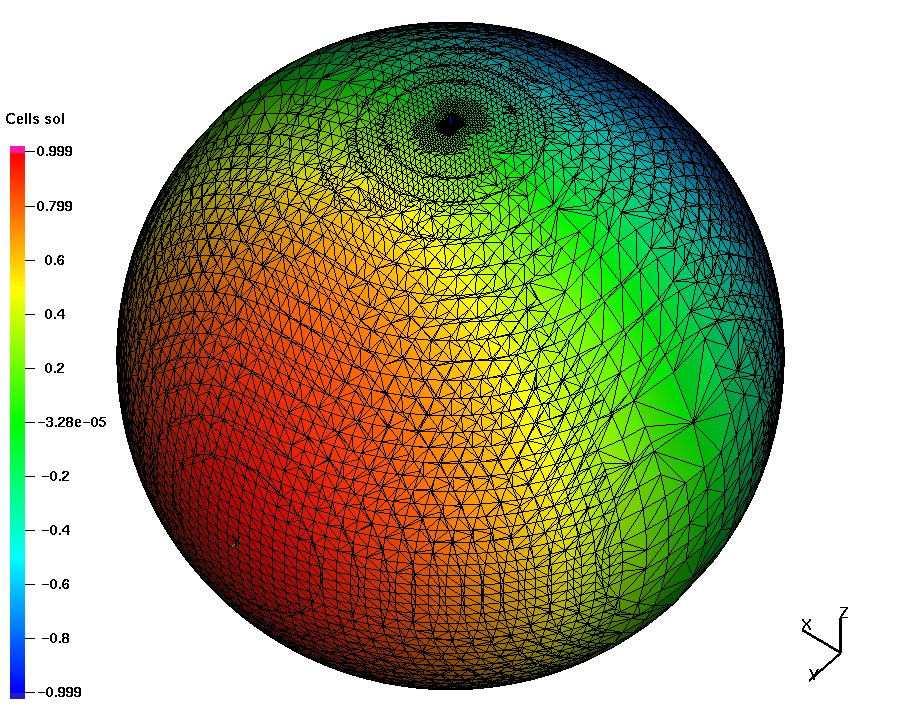}  \quad
  \includegraphics[width=0.25\textwidth]{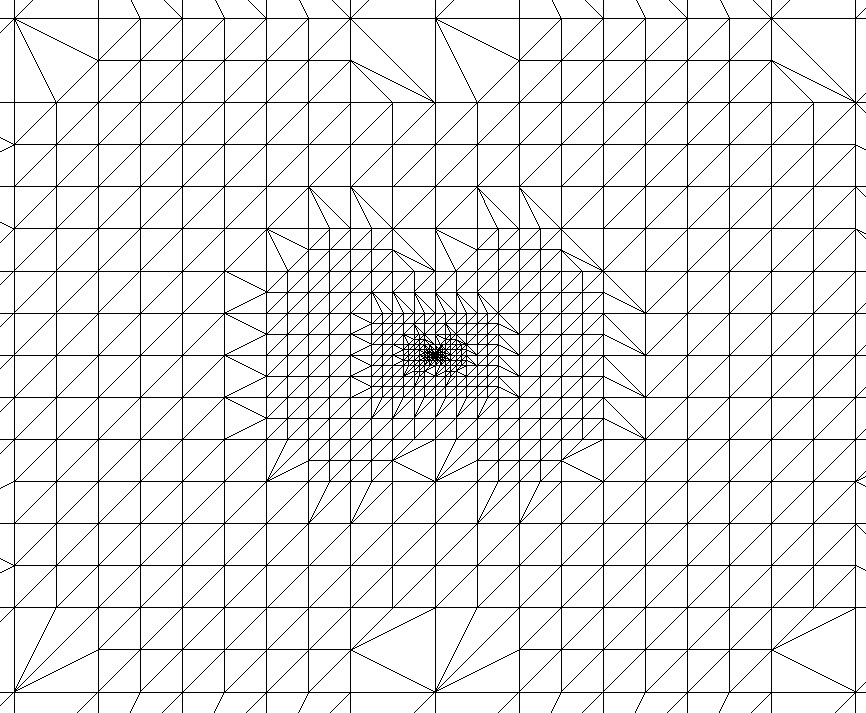}
\caption{\label{fig:lambda2} Left and middle plots display the cutaway of the bulk and trace surface meshes in Example~\ref{exam5} with $\lambda=0.6$ after 12 steps of refinement. The middle plot is shaded to reflect the numerical solution. The right picture show the 20-x zoom in of the surface mesh near the north pole.}
\end{center}
\end{figure}

Figure \ref{fig:lambda2} (left) displays a cutaway view which includes both the adaptively refined bulk and surface meshes. The meshes are shown after the 12 refinement steps.  The middle picture in Figure \ref{fig:lambda2} shows the surface
mesh superimposed on numerical solution. The coarsest mesh is in the regions with the smallest solution gradient.  Figure \ref{fig:lambda2} (right)  displays the surface mesh near the north pole magnified 20 times. This local mesh appears more structured, since the surface is locally close to a plane which cuts through a regular bulk mesh. 
\end{example}

\subsection{Convection-diffusion problem with an internal layer}

We now perform several tests for the advection-diffusion problem as in \eqref{problem}, with nonzero advection field $\bw$.
As usual, the properties of the problem essentially depend on the value of the dimensionless Peclet number. The Peclet number can be defined similar to volumetric case as $Pe=\frac{L W}{2\eps}$, where $L$ is a characteristic problem scale (say, the diameter of a closed surface $\Gamma$) and $W$ is a characteristic advection velocity. For low values of the Peclet numbers, the problem is close to the Laplace-Beltrami equation, while for higher  Peclet numbers, the problem may demonstrate behaviour typical to singular-perturbed equations, e.g., its solution may exhibit internal layers. The example considered below is chosen to illustrate the ability of the trace FEM to handle these different cases by employing computational tools developed for volumetric finite elements:
stablization, error indicators, and layer fitted meshes. Numerical results will show that the performance of such enhanced trace FEM appears to be similar to its volumetric counterparts applied to bulk advection-diffusion problems.

For the advection-diffusion problem we set weights $\alpha_r, \alpha_e$ in the error indicator \eqref{indicator} dependent on the
Peclet number as recommended in \cite{Verfurth2} for  the planar advection-diffusion problem:
\[
\alpha_r=\min\{\eps^{-1},h^{-2}_{S_T}\},\quad \alpha_e=\min\{\eps^{-1},h^{-1}_{S_T}\eps^{-\frac12}\}.
\]
In experiments below geometry does not play an important role and so we set $\alpha_g=0$.

\begin{example}\label{exam6}\rm
In this example, the stationary problem \eqref{problem} is solved on the unit sphere $\Gamma$, with
 the velocity field
$$
\mathbf{w}(\bx)=(-x_2\sqrt{1-x_3^2},x_1\sqrt{1-x_3^2},0)^T,
$$
which  is tangential to the sphere.
We set  $c = 1$ and consider $\varepsilon\in [10^{-6},1]$. Letting $W:=\|\mathbf{w}\|_{L^\infty(\Gamma)}$ and $L=\mbox{diam}(\Gamma)$, we compute the Peclet number as $Pe=\eps^{-1}$.

For the exact solution to \eqref{problem}, we take the function
$$
u(\mathbf{x})= {x_1 x_2}\mathrm{arctan}\left(\frac{2x_3}{\sqrt{\varepsilon}}\right).
$$
The corresponding right-hand side function $f$ is given
by
$$
f(\bx)=
\frac{12 \varepsilon^{3/2} x_1x_2x_3}{\varepsilon+4 x_3^2}+
\frac{16\varepsilon^{3/2}(1-x_3^2)x_1x_2x_3}{(\varepsilon+4 x_3^2)^2}
 + (6\varepsilon x_1x_2+\sqrt{x_1^2+x_2^2}(x_1^2-x_2^2))\mathrm{arctan}\left(\frac{2x_3}{\sqrt{\varepsilon}}\right)+u.
$$
When $\eps$ gets smaller, a layer  of the width $O(\eps^\frac12)$ is forming in $u$  along the equator of the sphere ($\{x_3=0\}\cap\Gamma$). The formation of characteristic  internal layers of $O(\eps^{\frac12})$-width is typical for advection-diffusion problems.

In this paper, we consider equations posed on closed surfaces, therefore we are not treating parabolic or exponential boundary layers.
\end{example}

\subsubsection{Lower Peclet number case}
\begin{figure}[ht!]
\begin{center}
  \includegraphics[width=0.45\textwidth]{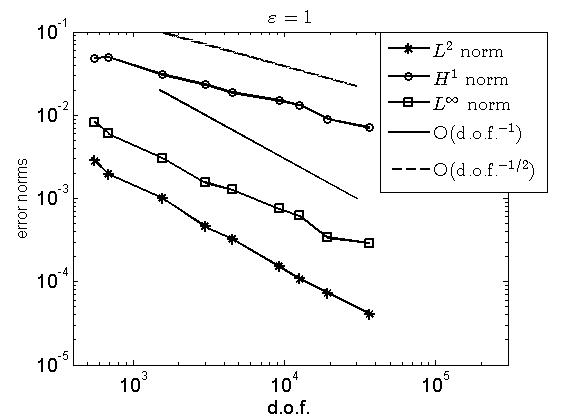}\qquad
  \includegraphics[width=0.45\textwidth]{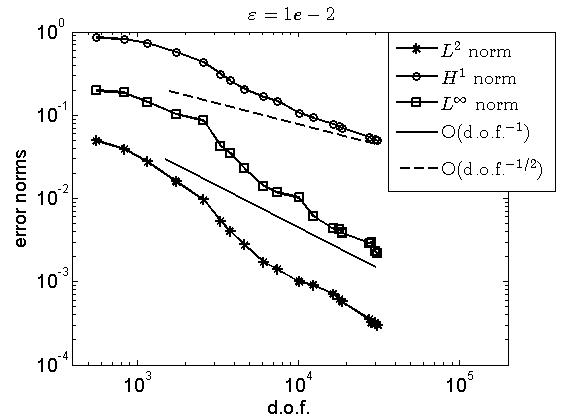}\\
\caption{\label{fig:conv1} Decrease of the error in the $L^2(\Gamma)$, $H^1(\Gamma)$ and $L^\infty(\Gamma)$ norms  for the advection-diffusion problem from Example~\ref{exam6} with  $Pe=1$ (left) and $Pe=100$ (right). The problem is solved on a sequence of adaptively refined grids.}
\end{center}
\end{figure}

\begin{figure}[ht!]
\begin{center}
  \includegraphics[width=0.45\textwidth]{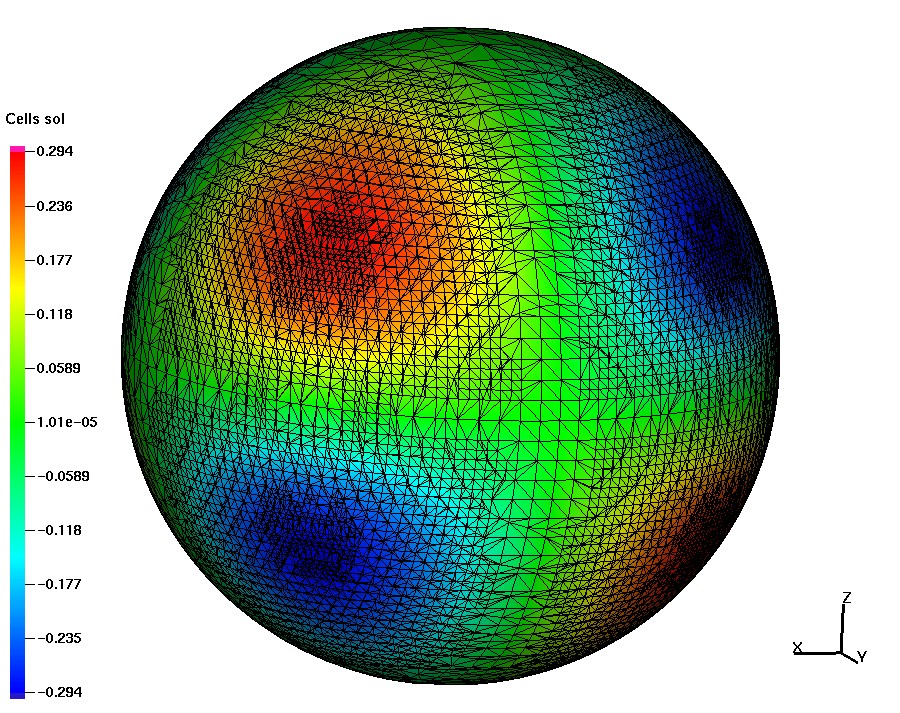}\qquad
  \includegraphics[width=0.45\textwidth]{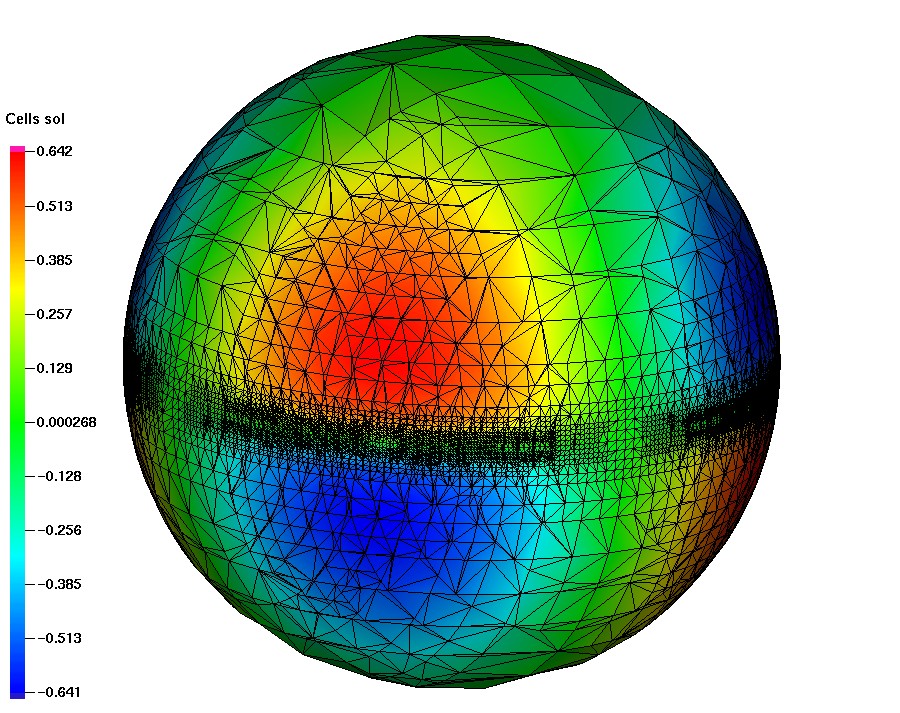}
\caption{\label{fig:conv1.5} The numerical solutions  for the convection-diffusion problem from Example~\ref{exam6} with $Pe=1$ (left) and $Pe=100$ (right). The surface meshes are shown after 7 (left) and 9 (right) steps of adaptive refinement.}
\end{center}
\end{figure}

We consider  the equation \eqref{problem} for  $Pe\le 100$ to be  non-singular perturbed. Hence, for this case of lower Peclet numbers, we expect the trace finite element to  behave similar to the case of Laplace-Beltrami equations with smooth solution. We add no stabilization in this case
and recover the expected $O(h^2)$ and $O(h)$ convergence rates on a sequence of uniformly refined grids in $L^2$ and $H^2$ surface norms,
respectively (not shown). Figure~\ref{fig:conv1} shows the error reduction plots for $Pe=1$ and $Pe=100$ if a mesh adaptation is performed based on the error indicator~\eqref{indicator}. For $Pe=100$, the error norms are approximately one order  bigger than for $Pe=1$, but  in both cases the convergence curves demonstrate optimal reduction rates versus the number of active degrees of freedom. The numerical solutions  and the surface meshes after several steps of adaptive refinement are displayed in Figure~\ref{fig:conv1.5}. For $Pe=100$, most of refinement happens closer to the equator of the sphere.

The results are shown for the trace finite element formulation with surface gradient \eqref{FEM_fg}. The results for the method with full gradient \eqref{FEM_fg} were largely similar.

\subsubsection{Higher Peclet number case}
Now we consider the equation \eqref{problem} and  $Pe\ge 10^3$.  For the growing Peclet number, the problem becomes increasingly  singular perturbed. For $Pe=10^3$ the internal layer becomes pronounced. Now and in all further experiments in this section, we apply SUPG stabilized trace finite element method~\eqref{FEM_SUPG}.

\begin{figure}[ht!]
\begin{center}
  \includegraphics[width=0.45\textwidth]{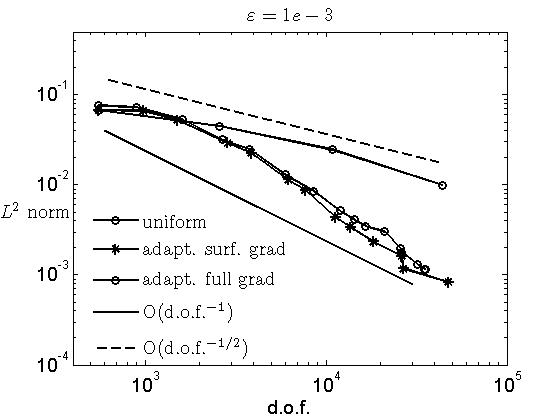}\qquad
  \includegraphics[width=0.45\textwidth]{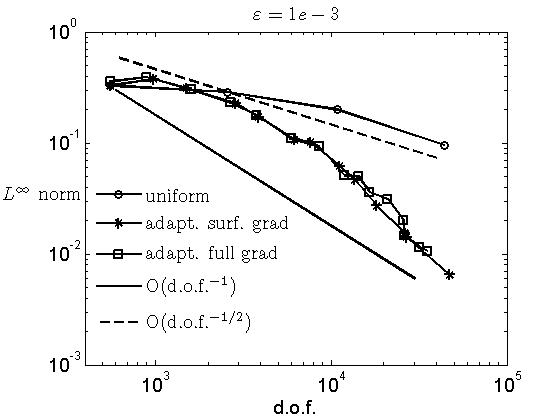}\\
\caption{\label{fig:conv2} Decrease of the  error for the ``surface gradient'' and ``full gradient''
versions of the trace FEM for the advection-diffusion problem from Example~\ref{exam6} with $Pe=10^3$. The error is measured in the $L^2(\Gamma_h)$ norm (left) and $L^\infty(\Gamma_h)$ norm (right). Both plots
compare the error of  the ``surface gradient'' and ``full gradient'' variants  of the method on a sequence of adaptively refined grids with the error of the ``surface gradient'' variant on the sequence of uniformly refined meshes.}
\end{center}
\end{figure}

Already for $Pe=10^3$ the uniform refinement doesn't lead to typical $O(h^2)$ convergence behavior
(unless the mesh is sufficiently fine to resolve the boundary layer.)
This is well seen from results shown in Figure~\ref{fig:conv2}. The error decrease on uniformly refined grid
is shown only for the ``surface gradient'' variant of the trace FEM \eqref{FEM}, since the results with \eqref{FEM_fg} were very similar. The adaptive refinement based on the indicator \eqref{indicator} leads, however, to the optimal error decrease with respect to the number of active d.o.f. These results are also demonstrated in Figure~\ref{fig:conv2}. The ``surface gradient'' and ``full gradient'' variants show very close results.

\begin{figure}[ht!]
\begin{center}
  \includegraphics[width=0.45\textwidth]{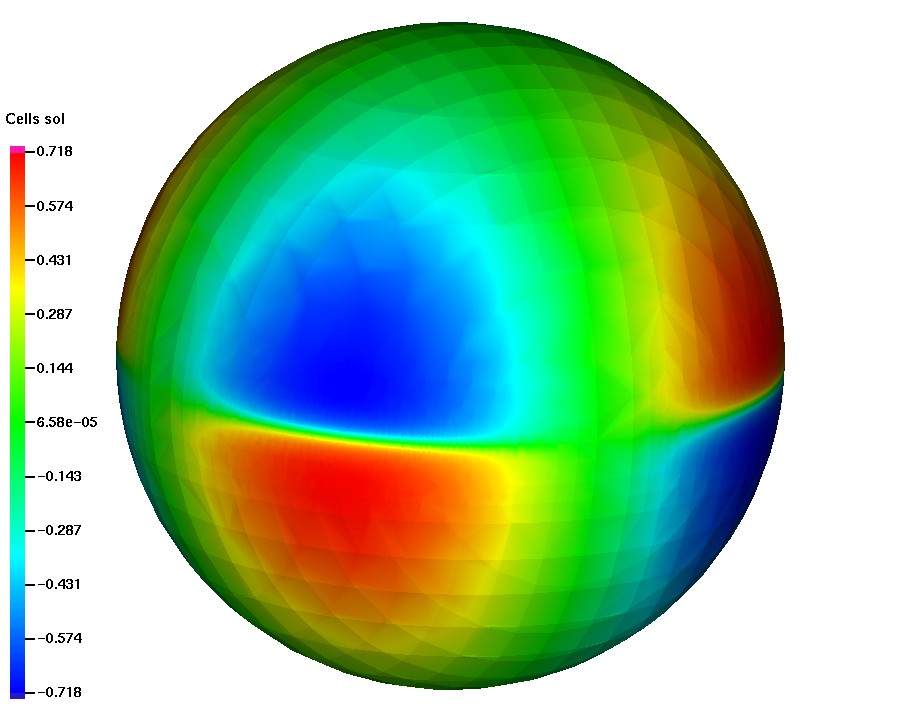}\qquad
  \includegraphics[width=0.45\textwidth]{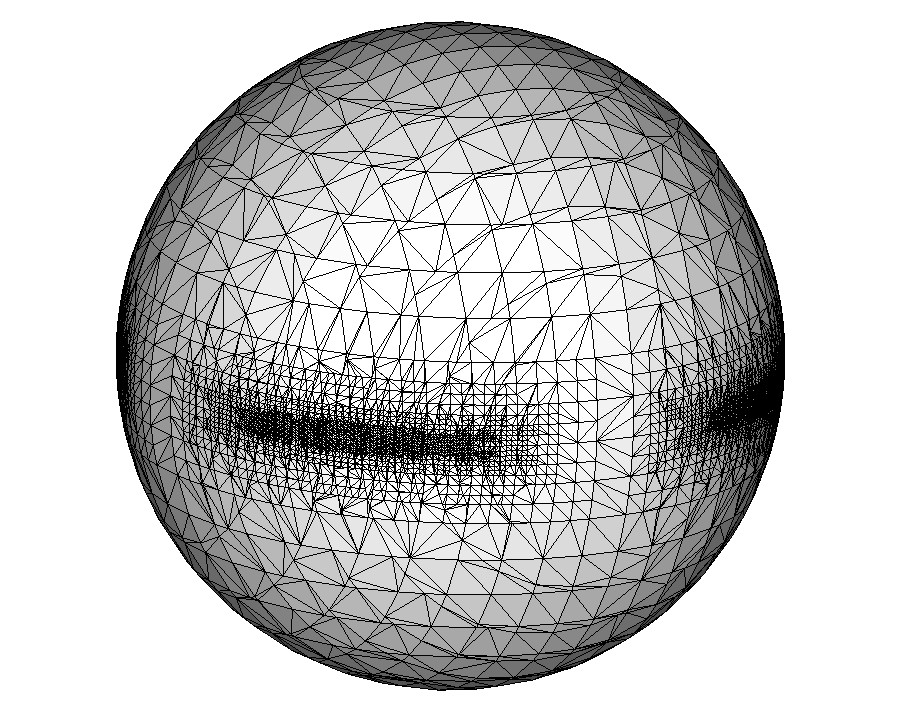}
\caption{\label{fig:conv3} Numerical solution and the mesh after 9 refinement steps for the convection-diffusion problem from Example~\ref{exam6} with $Pe=10^3$.}
\end{center}
\end{figure}

The numerical solution computed after 9 refinement steps and the corresponding surface mesh are demonstrated in Figure~\ref{fig:conv3}. The internal characteristic layer is well seen in the solution. The error indicator~\eqref{indicator} enforces an aggressive refinement in the regions of the layer.

\begin{figure}[ht!]
\begin{center}
  \includegraphics[width=0.45\textwidth]{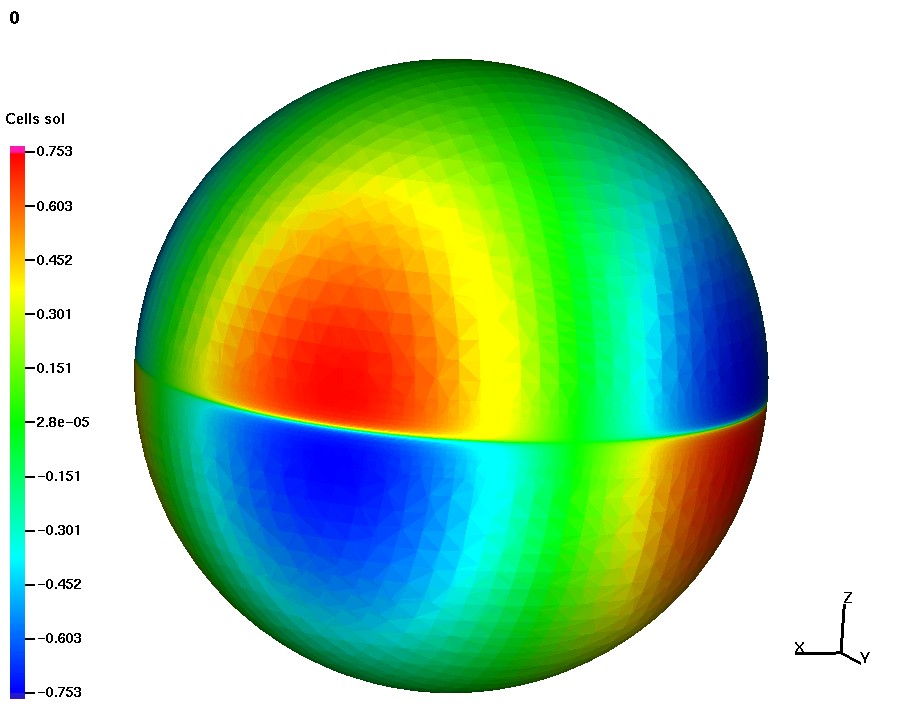}\qquad
  \includegraphics[width=0.45\textwidth]{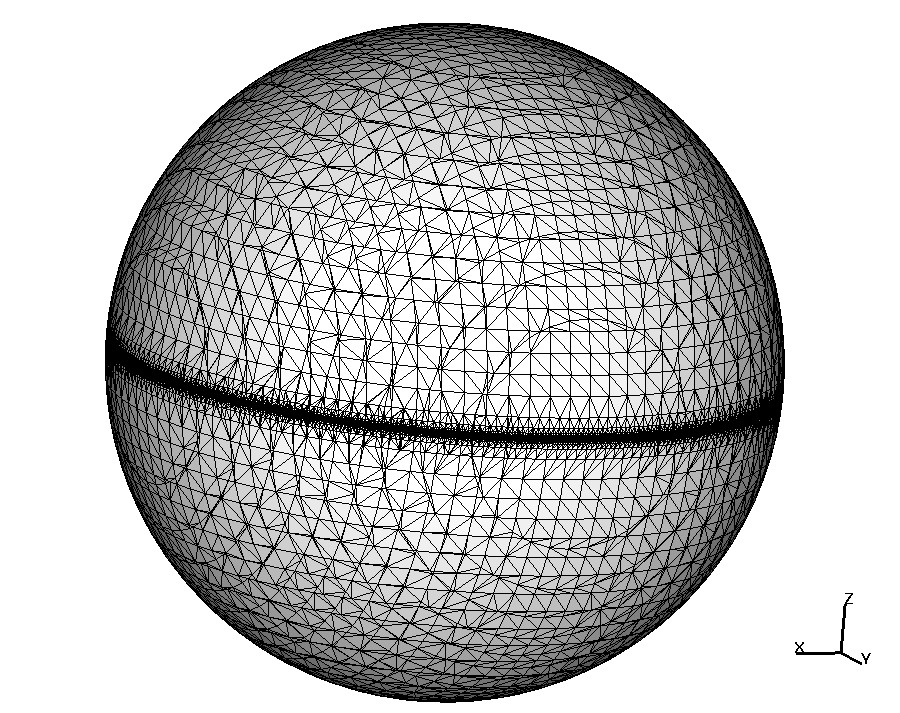}
\caption{\label{fig:conv4} Numerical solution and the Shishkin mesh for the advection-diffusion problem from Example~\ref{exam6} with the $Pe=10^4$.}
\end{center}
\end{figure}

Further, we consider the same problem  with the higher Peclet number equal to $10^{4}$.
This time, the adaptive refinement based on the error indicator  was not found to produce
optimal error reduction for the number of degrees of freedom up to 50000. Hence, we consider this problem to be a good test case for  layer fitted meshes. We choose \textit{Shishkin meshes} as one of the best studied class of meshes for singular-perturbed volumetric or planar problems. Shishkin meshes require an a priori knowledge of where a layer occurs and provide optimal convergence with respect to the total number of degrees of freedom~\cite{Shishkin}.

Let $N$ be the total number of nodal degrees of freedom available to discretize a singular-perturbed convection-diffusion problem. Assume that the solution to the problem has only an internal characteristic layer. Then to build a Shishkin mesh one defines a narrow band of width $O(\sqrt{\eps}\ln N)$ around the layer and 
considers a mesh which is uniform inside and outside the band and contains $O(N)$ nodes in the interior of the narrow band as well as in the rest of computational domain (see, e.g., ~\cite{Shishkin} for accurate definitions).
We extend this construction to the case of octree bulk meshes and the singular-perturbed problem
posed on a surface as follows. We build an initial octree mesh such that $h_{\min}=1/128$ was the size of cubes inside the strip $|x_3|\le 1/64$ (this defines our ``narrow band'' containing the layer). In the rest of the bulk domain, the grid was aggressively coarsened up to $h_{\max}=1/4$. The resulted number of active degrees of freedom for this initial mesh was $10356$.
Further, the mesh was uniformly refined two times, leading to layer fitted meshes with $22830$ and $101332$
active degrees of freedom. We note two deviations of our construction  from the classical notion of a Shishkin mesh: (i) We do not re-balance the mesh to account for the logarithmic factor in the width of a narrow band of a canonical Shishkin mesh; (ii) The mesh outside our narrow band is not completely uniform due to a transition region, which is necessary to keep the octree balanced (two neighboring cubes may differ in size at most by a factor of 2).

\begin{table}
\begin{center}
\caption{Convergence of numerical solutions to the advection-diffusion problem from Example~\ref{exam6}, with $Pe=10^4$,
 on a sequence of Shishkin meshes. \label{tab2}}\smallskip
\small
\begin{tabular}{r|llllll}\hline
\#d.o.f. & $L^2$-norm & rate & $H^1$-norm& rate & $L^\infty$-norm& rate \\ \hline\\[-2ex]
10356&	      4.870e-3 &      & 1.577e-0 &      & 6.725e-2 & \\
22830&	      1.428e-3 & 1.77 & 7.597e-1 & 1.05 & 1.718e-2 & 1.97\\
101332&   	  3.739e-4 & 1.93 & 3.761e-1 & 1.01 & 5.484e-3 & 1.65 \\  \hline
\end{tabular}
\end{center}
\end{table}

Figure~\ref{fig:conv4} shows the numerical solution computed on the finest surface Shishkin mesh.
Note that the internal layer is sharp and resolved. We observe no numerical oscillations in a vicinity of the layer. Table~\ref{tab2} presents the norms of the finite element error  on the sequence of the layer fitted meshes and corresponding convergence factors.

\begin{figure}[ht!]
\begin{center}
\includegraphics[width=0.45\textwidth]{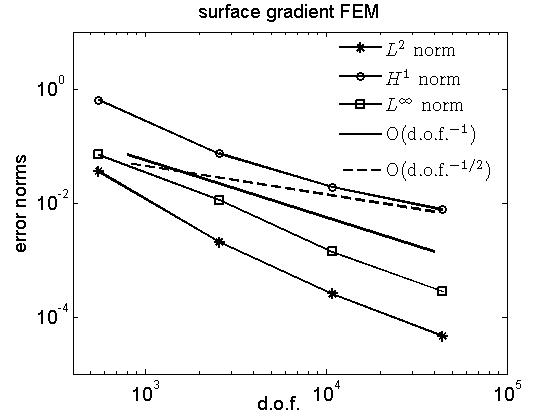}\qquad  \includegraphics[width=0.45\textwidth]{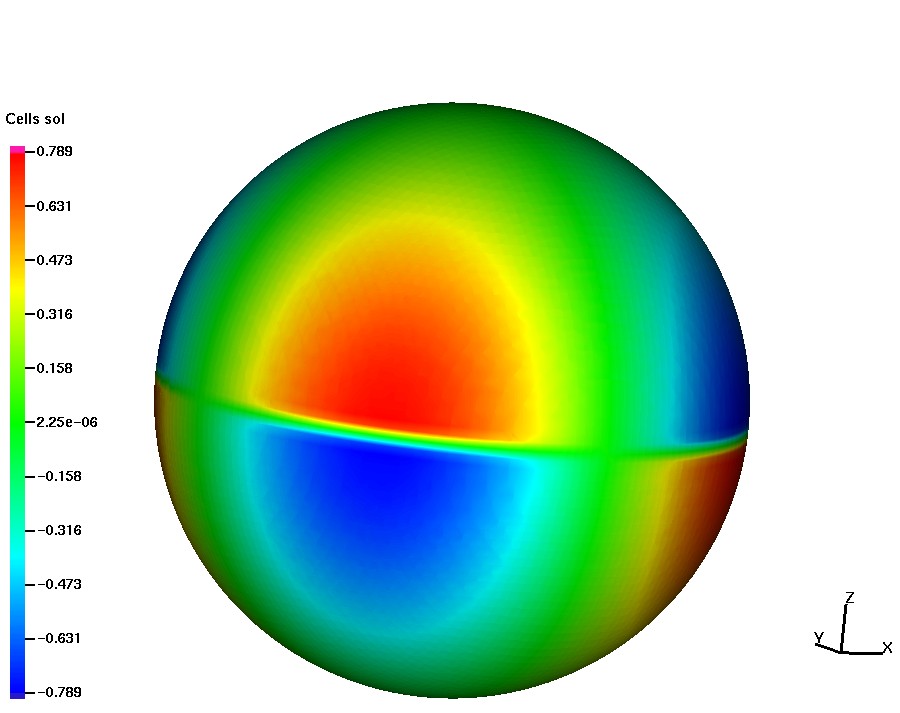}
\caption{\label{fig:conv6}Left: Decrease of the error in the $L^2(\Gamma_{ext})$, $H^1(\Gamma_{ext})$ and $L^\infty(\Gamma_{ext})$ norms  for the advection-diffusion problem from Example~\ref{exam6} with $Pe=10^{6}$. $\Gamma_{ext}$ is a part of the domain well separated from the internal layer.  The problem was solved on a sequence of uniformly refined meshes.   Right: Numerical solution for the advection-diffusion problem from Example~\ref{exam6} with the $Pe=10^6$ on a uniformly refined mesh. The internal layer is not resolved and clearly smeared over few mesh sizes. No spurious oscillations can be noted.}
\end{center}
\end{figure}

Finally, we solve the same problem, but now with $\eps=10^{-6}$, leading to $Pe=10^6$. The width of the internal layer is $O(\eps^{\frac12})$, and we are not attempting to resolve it with  a layer fitted mesh.
Instead we solve the problem on a sequence of uniformly refined grids. Since we used the SUPG stabilized formulation of the finite element method, we can expect that similar to a planar or a  volumetric cases
the layer would be smeared,  numerical oscillation damped, and finite element solution converges to the exact one
outside the layer. This is exactly what we observed for our trace finite element method. Thus, Figure~\ref{fig:conv6} demonstrates the computed solution as well as the finite element error decrease in a
part of the domain separated from the layer: All norms in Figure~\ref{fig:conv6} (left) were computed over
part of the sphere $\Gamma_{ext}:=\{\bx\in\Gamma\,:\, |x_3|>0.3\}$.

\section{Conclusions} \label{s_concl}

We studied a trace finite element method for partial differential equations
posed on hypersurfaces in $\mathbb{R}^3$. An extension to curves in  $\mathbb{R}^2$ is straightforward.
The paper demonstrates that using such standard computational
tools  as cartesian octree grids, a marching cubes method, and trilinear bulk finite elements leads to a
second order accurate method with a number of attractive features: The mesh
is unfitted to a surface; One uses standard finite element tools without any extension of equation from the surface to bulk domain; The method works for surfaces defined implicitly; The number of
active d.o.f. is optimal and comparable to methods in which $\Gamma$ is meshed directly; Optimal order of convergence in $H^1$ and $L^2$ norms is proved for quasi-uniform bulk grids. Moreover, due to the natural connection
to bulk elements, many tools and techniques well established for ``usual''  discretizations carry over to the surface case. In this paper, we experimented with several such techniques: adaptivity based on an error indicator, SUPG stabilization
for transport dominant equations, and layer fitted meshes. Numerical analysis supports experimental observations.
In forthcoming papers we plan to extend the adaptive octree trace finite element method to PDEs defined on evolving surfaces and 
 to apply it for the simulation of  flow and transport in fracture media.

\subsection*{Acknowledgments}
This work has been supported  by RFBR through the grants 12-01-33084, 14-01-00731, 14-01-00830 and by NSF through the Division of Mathematical Sciences grant 1315993.

\bibliographystyle{elsarticle-num}
\bibliography{ChernOlshan}
\end{document}